\documentclass[a4paper,12pt]{article}
\usepackage{xcolor} 
\usepackage{epsfig}
\usepackage{amsmath}
\usepackage{amssymb}
\usepackage{amsthm}
\usepackage{mathtools}
\usepackage[top=2.5cm, bottom=2.5cm, left=1.5cm, right=1.5cm]{geometry}
\usepackage{microtype}
\usepackage[utf8]{inputenc}
\usepackage{float}
\usepackage{hyperref}
\usepackage{orcidlink}
\usepackage{bera}
\newcommand{\ds}{\displaystyle }

\newtheorem{thm}{Theorem}[section]
\newtheorem{Proposition}[thm]{Proposition}

\newtheorem{rmk}[thm]{Remark} 

\newcommand{\beq}{\begin{equation}}
\newcommand{\eeq}{\end{equation}}
\begin{document}
\author{Thomas M. Michelitsch  \orcidlink{0000-0001-7955-6666} $^a$, Giuseppe D'Onofrio \orcidlink{0000-0002-6155-3114} $^b$, \\ Federico Polito \orcidlink{0000-0003-1971-214X} $^c$, Alejandro P. Riascos \orcidlink{0000-0002-9243-3246} $^d$
\\[1ex]
\footnotesize{$^a$ Sorbonne Universit\'e, Institut Jean le Rond d’Alembert,
CNRS UMR 7190} \\
\footnotesize{4 place Jussieu, 75252 Paris cedex 05, France} \\
\footnotesize{E-mail: michel@lmm.jussieu.fr} \,\, (corresponding) \\[1ex]
\footnotesize{$^b$ Department of Mathematical Sciences, Politecnico di Torino, 10129 Torino, Italy}  \\
\footnotesize{E-mail: giuseppe.donofrio@polito.it}\\[1ex]
\footnotesize{$^c$ Department of Mathematics ``G.\ Peano'', University of Torino, 10123 Torino, Italy}  \\
\footnotesize{E-mail: federico.polito@unito.it}\\[1ex]
\footnotesize{$^d$ Departamento de Física, Universidad Nacional de Colombia, Bogot\'a, Colombia}  \\
\footnotesize{E-mail: alperezri@unal.edu.co} 
}
\title{Random walks with stochastic resetting in complex networks: a discrete time approach}
\maketitle
\begin{abstract}
We consider a discrete-time Markovian random walk with resets on a connected undirected network. The resets, in which the walker is relocated to randomly chosen nodes, are governed by  an independent discrete-time renewal process. Some nodes of the network are target nodes, and we focus on the statistics of first hitting of these nodes. In the non-Markov case of the renewal process, we consider both light- and fat-tailed inter-reset distributions.
We derive the propagator matrix in terms of discrete backward recurrence time PDFs and in the light-tailed case
we show the existence of a non-equilibrium steady state.
In order to tackle the non-Markov scenario, we derive
a defective propagator matrix
which describes an auxiliary walk characterized by killing the walker as soon as it hits target nodes. This propagator provides the information on the mean first passage statistics to the target nodes. 
We establish sufficient conditions for ergodicity of the walk under resetting. 
Furthermore, we discuss a generic resetting mechanism for which the walk is non-ergodic.
Finally, we analyze inter-reset time distributions with infinite mean where we focus on the Sibuya case.
We apply these results to study the mean first passage times for Markovian and non-Markovian (Sibuya) renewal resetting protocols in realizations of Watts-Strogatz and Barabási–Albert random graphs.
We show non trivial behavior of the dependence of the mean first passage time on the proportions of the relocation nodes, target nodes and of the resetting rates. It turns out that, in the large-world case of the Watts-Strogatz graph, the efficiency of a random searcher particularly benefits from the presence of resets.
\end{abstract}
%
%
%
%


\section*{Lead paragraph}

{\it 
Dynamics with stochastic resetting (SR) occurs whenever the time evolution of a phenomenon is characterized by repeated relocations that happen randomly in time, according to a certain mechanism.
SR is ubiquitous in nature and society: in foraging, an animal undertakes repeated excursions from its lair to search for food;
in biochemistry, when certain biomolecules such as proteins are searching for binding sites; problem-solving strategies; financial markets recurrently hit by crises. Resets may represent catastrophic events such as earthquakes, volcanic eruptions or wood fires after which flora and fauna restart to develop, or seasonal hurricanes forcing human societies to reconstruct infrastructures.
The first three mentioned examples fall into the category of \emph{random search with resetting}. In this respect, 
a major issue is whether random search strategies are improved when combining them with resetting.
In many cases, resetting can indeed significantly reduce the time for reaching a target
and it has become a matter of fundamental interest to model such resetting strategies. 

In this paper we consider the behavior of a discrete-time Markov walker moving randomly on discrete structures, where we focus
in particular on undirected connected graphs.
This random motion is subjected to successive recurrent resets (instantaneous relocations to nodes) governed by a counting process which can be of Markovian or non-Markovian nature. 
We analyze various resetting mechanisms and characterize these choices by means of a relocation matrix.
We explore the resulting dynamics on realizations of the Watts-Strogatz and Barabási-Albert random graphs and investigate how resets can affect the efficiency in the searching of a target. We derive an exact formula of the propagator of the associated walk where it turns out that if the waiting time between consecutive resets has a finite mean, 
the infinite time limit of the propagator matrix is a non-equilibrium steady state (NESS).

Interesting features of the dynamics are revealed by the hitting time statistics for the reaching of target nodes and asymptotic behaviors.
In particular, we analyze the mean first passage time for non-Markovian resetting, where we focus on the case of (fat-tailed) Sibuya distributed inter-resetting times.
}

\section{Introduction}
\label{Intro}

The theory of SR made its first appearance in the literature only a decade ago \cite{Evans-et-al2011} and has attracted considerable attention since then.
In that paper, a Brownian particle is relocated (or reset) to the starting point and the time between consecutive relocations are independent and exponentially distributed, which causes the finiteness of the mean time to reach a target.
Subsequent developments concerned different resetting mechanisms, different underlying stochastic processes, and other aspects (see the review \cite{Evans-etal2020}).
Several models in the literature are devoted to Markovian resets
\cite{Das2022,Boyer-et-al2023,Gupta-etal-Frontiers2022,Pal2016},
while others relax the Markov property \cite{Evans-etal2020,Bodrovka-et-al2019a,Bodrovka-et-al2019b, Sandev-et-al2022,Sandev-Mendez-et-al2023, Barkai-etal-2023}.
In both cases, processes other than the Brownian motion have also been analyzed \cite{Bodrovka-et-al2019a, Sandev-et-al2022, Jalakovsky-et-al2023}. An aspect of particular interest considered in the literature, and in the present paper as well, is that of random search (see \cite{PalSandev2023,Chechkin2018,Benichou_et_al2011,Eule-et-al2016,GamlmesCamosMendez2020,VasquezBoyer2021,MercadoVasquez2022,PalReuveni2017} just to name a few). In these models, a particle moves randomly in a certain environment looking for one or multiple targets to visit. The first passage time (we use synonymously the term ``first hitting time'') of the particle to one of the targets is a measure of the efficiency of searching strategies.
First passage features of long-range motions such as L\'evy flights and walks have been investigated for a long time \cite{MetzlerKlafter2000,Palyulin2019,RiascosMateos2012,MiRia2017} and the effects of resetting on these motions were extensively studied 
\cite{Kusmierz2014,ZbikDybiek2024}. 
Furthermore, Markovian resetting of random walks in complex graphs was investigated recently \cite{Riascos_boyer_stat_res2020} and the dependence of the resetting probability on the mean first passage times has been analyzed in order to define an optimal search strategy.
SR under certain conditions gives rise to the emergence of non-equilibrium steady states (NESS) \cite{Evans-et-al2011}. For instance, a free Brownian particle becomes localized around the resetting site and stationary under Markovian resetting (consult
\cite{Sandev-et-al2022,Sandev-Mendez-et-al2023,Barkai-etal-2023} for extensive discussions). 
While resetting to a deterministic location has been analyzed thoroughly, resetting to random locations is much less explored \cite{Boyer-et-al2023,Riascos-Boyer-etal2021}. 

An interesting class of problems emerges in presence of ``partial resetting'', where a particle is reset from its actual position, say $x$, to a new position $x'=ax$ where $a=0$ corresponds to complete reset and $a=1$ to no reset. Brownian particles under Markovian partial resetting are analyzed in \cite{Tal-Friedman-etal2022}. In particular, they showed that for $0\leq a < 1$ a NESS always emerges. The effects of partial resetting have been studied recently in different directions \cite{Dahelenburg-Metzler-et-al2021, van_der_hof2023, don_2024, Pierce2022, Biroli-et-al2024}. 

A further aspect worthy of consideration is the possibility that the reset gives rise to a relocation to a random position chosen from a set with some specified probability law. For instance, in  \cite{BoyerSolisSalas2014,MeyerRieger2021} the particle is preferentially reset to previously visited sites or more in general has a memory of previously visited locations.
A discrete time renewal approach was employed to analyze Markov random walks in networks with certain time dependent resetting mechanisms \cite{ChenYe2022}, focusing on non-Markov effects of aging of the walker and first passage features for a single target node.

The present work concerns discrete time dynamics on discrete structures such as lattices and realizations of some random graphs in presence of resetting governed by Markov and non-Markov counting processes.
In particular, we describe the behavior of a Markov random walker
navigating in discrete time in an undirected network and which is reset to a randomly chosen
node at the arrival time instants of the mentioned discrete-time renewal process.
We characterize the random choice of the node where the walker is relocated in a reset by a specific relocation matrix.
Taking into consideration the statistics of the discrete backward recurrence times, we are able to derive a compact formula for the propagator matrix
by solving a specific renewal equation.
If the waiting time between the resets has a finite mean, the infinite time limit of the propagator matrix is a NESS.

The present paper is organized as follows. After a brief summary of the properties of discrete-time renewal processes, we study Markovian resetting in Section \ref{Markovian_reseting}, that is the case in which the renewal process governing the resets is a Bernoulli process. We focus on first passage quantities such as the mean first passage time (MFPT) the walker takes to reach specific target nodes.
We also investigate numerically the dependence of the Bernoulli parameter $p$ (coinciding with the resetting rate) for various choices of the resetting nodes (r-nodes)\footnote{We refer to ``resetting nodes'' if the walker can be relocated to them.}. The motivation for this stems from the question on whether the efficiency of a random searcher is improved by changing the resetting scenario.
In this direction, we illustrate our results by means of numerical simulations in Watts-Strogatz and Barabási-Albert networks. A result of interest is that, for a fixed value of the resetting rate, the efficiency has a non monotonic behavior with respect to the proportion of the resetting nodes. Conversely, for a fixed number of r-nodes, in some cases an optimal resetting rate can be found to minimize the MFPT.
In Section \ref{FPH-quantities} we explore the first hitting statistics for walks 
under non-Markovian renewal resets.
To that end, we
endow the target nodes with a killing feature: the walker is killed (removed from the network) as soon as it reaches one of the target nodes. This simple assumption allows us to derive a modified propagator matrix (the ``survival propagator'') which is linked to killed sample paths and gives information on the first hitting statistics. With these results, we establish sufficient conditions for which the walk in presence of resetting is ergodic. 
Moreover, in Section \ref{5825}, we investigate a generic resetting mechanism for which the walker cannot explore the whole network implying a lack of ergodicity.
As a further application, we explore in Section \ref{Sibuya_ARW} the first hitting dynamics to target nodes for the non-Markovian case of Sibuya distributed time intervals between resets. 

\section{Markovian walks under SR}
Here we elaborate some technical details of discrete-time renewal processes and occupation time statistics relevant for motions in networks with SR. The associated discrete-time processes are much less common in the literature compared to their continuous time counterparts. For further information and applications of discrete-time renewal processes, continuous-space-time scaling limits, and connections with discrete semi-Markov chains consult \cite{PachonPolitoRicciuti2021,Barbu_Limnios-et-al2008,MichelitschPolitoRiascos2021,ADTRW2021,SRW2023,GSD-Squirrel-walk2023}, 
and see \cite{GodrecheLuck2001} for related occupation time statistics in continuous time.
In this paper we consider a discrete-time  Markovian random walker moving on a finite, undirected, connected, aperiodic, and ergodic graph \cite{Barabasi2016,Newman2010}. The walker is subjected to recurrent resets governed by a discrete-time renewal process.
To this aim, we first recall a Markovian random walk.

At each integer time $t\in \mathbb{N}_0=\{0,1,2,\ldots\}$ the walker moves from one to another connected node. We denote the nodes by $i=1,\ldots, N$ and characterize the topology of the finite, undirected network by the symmetric adjacency matrix $\mathbf{A}$, with $A_{ij} = A_{ji}=1$ if the nodes $i$ and $j$ are connected by an edge, and $A_{ij}=0$ otherwise. We avoid self loops by setting $A_{ii}=0$.  Further,  $K_i= \sum_{j=1}^NA_{ij}$ is the degree of a node $i$ which corresponds to the number of its neighbors. The steps are drawn from the transition matrix 
$\mathbf{W}$ where its elements are $W_{ij}=\frac{A_{ij}}{K_i}$ \cite{Barabasi2016,fractionalbook-2019,NohRieger2004}.
We point out that in undirected networks, $\mathbf{W}$ is not symmetric, unless the degrees $K_i$ coincide.
The propagator matrix, that is the transition matrix for $t$ steps, $\mathbf{P}^{(0)}(t) = [P^{(0)}_{ij}(t)]$ fulfills the master equation
\beq
\label{master_eq}
P^{(0)}_{ij}(t+1) =  \sum_{k=1}^N P^{(0)}_{ik}(t)W_{kj} , \hspace{1cm}  P^{(0)}_{ij}(0) =\delta_{ij}, \hspace{1cm} t \in \mathbb{N}_0,
\eeq
where $P_{ij}^{(0)}(t)$ denotes the probability that the walker occupies $j$ at time $t$, if it started from node $i$ at time $0$.
The resulting walk is a Markov chain with transition matrix 
\beq
\label{purely_markovian}
\mathbf{P}^{(0)}(t) = \mathbf{W}^t = |\phi_1\rangle \langle \bar \phi_1| + \sum_{m=2}^N  \lambda_m^t |\phi_m\rangle \langle {\bar \phi}_m| , \hspace{1cm} \mathbf{P}^{(0)}(0) = \mathbf{1},
\eeq
where we used Dirac's bra-ket notation to make evident the spectral properties of $\mathbf{W}$. Note that, since we consider undirected networks, $\mathbf{W}$ has $N-1$ real eigenvalues $|\lambda_m|< 1$  and one largest single  Perron–Frobenius eigenvalue $\lambda_1=1$ \cite{Newman2010,fractionalbook-2019,VanMiegem2011,Meyer2000}. 
This assumption implies that we assume the presence of at least one return path of odd length \cite{NohRieger2004}. This, in turn, excludes that the graph is bipartite \cite{lovasz}.
In (\ref{purely_markovian}) we have introduced the right and left eigenvectors $|\phi_r\rangle$, $\langle {\bar \phi}_s|$ of $\mathbf{W}$, respectively, where $\mathbf{1} =\sum_{m=1}^N |\phi_m\rangle \langle {\bar \phi}_m|$, and
$ \langle \phi_m|{\bar \phi}_n\rangle =\delta_{mn}$. We denote with $\langle i|\phi_m \rangle$ the $i$-th Cartesian component of the
column vector $|\phi_m \rangle$ and with
$\langle {\bar \phi}_m|j \rangle$ the $j$-th Cartesian component of the row vector $\langle {\bar \phi}_m|$.
By construction, the transition matrices $\mathbf{W}^t$ are row-stochastic inheriting this feature from the one-step transition matrix $\mathbf{W}$. The first term in (\ref{purely_markovian}) is the stationary distribution in which
$ \langle i|\phi_1\rangle \langle {\bar \phi}_1| j\rangle = {K_j}/{\sum_{r=1}^N K_r}$  is
independent of $i$ \cite{Newman2010,NohRieger2004}. The initial occupation distribution defined by the state (row-) vector $\langle p(0)|$ evolves into  $\langle p(t)| = \langle p(0)|\mathbf{P}^{(0)}(t)$
which is  a row vector containing the components $p_j(t)$ of the occupation probabilities of the nodes $j=1,\ldots,N$.

\medskip

Now, we assume that the above Markovian walk is subjected to resetting. When a reset occurs, the walker is relocated to any of the nodes of the network according to some probability distribution that we will specify later on. 
The resetting times $J_n \in \mathbb{N}$, $n \in \mathbb{N}$, correspond to the arrival times of the discrete-time renewal process (independent of the steps)
\beq
\label{renewal_times}
J_n =\sum_{r=1}^n \Delta t_r, \hspace{1cm} J_0=0,  \hspace{1cm} \Delta t_r \in \mathbb N=\{ 1, 2, \ldots \}, 
\eeq
where the $\Delta t_j$ are IID random variables describing time intervals between consecutive resets. 
We call (\ref{renewal_times}) the renewal resetting process (RRP).
We introduce the counting renewal process
\beq
\label{counting_var}
{\cal N}(t) = \max(n\geq 0: J_n \leq t)  , \hspace{1cm} {\cal N}(0)=0 , \hspace{1cm} t  \in \mathbb{N}_0,
\eeq
counting the number of resets up to and including time $t$. We assume that there is no reset at $t=0$. The time intervals $\Delta t_j \geq 1$ follow the discrete PDF
\beq
\label{discret-PFD}
\mathbb{P}[\Delta t = r] =\psi(r) , \hspace{1cm} r\in \mathbb{N},
\eeq
with $\psi(0)=0$ to ensure $\Delta t_j  \geq 1$. We will need its generating function (GF)
\beq
\label{Gf_spi}
{\bar \psi}(u) = \big\langle u^{\Delta t} \big\rangle = \sum_{r=1}^\infty \psi(r) u^r, \hspace{1cm} |u| \leq 1,
\eeq
where ${\bar \psi}(1)=1$ implies the normalization of $\psi(t)$.
It is also useful to consider the state probabilities of (\ref{counting_var})
\beq
\label{state_probs}
\mathbb{P}[\mathcal{N}(t) = n] = \Phi^{(n)}(t) = \big\langle \Theta(J_n,t,J_{n+1}) \big\rangle.
\eeq
In \eqref{state_probs} we make use of the indicator function
$\Theta(J_n,t,J_{n+1})$,
where, for $a,b,t \in \mathbb{N}_0$, $a<b$,
\begin{align}
\label{Theta-a-b}
\Theta(a,t,b) &=\Theta(t-a)-\Theta(t-b) 
= \left\{\begin{array}{l} 1, \hspace{1cm} {\rm for} \hspace{0.5cm} a \leq t \leq b-1,  \\[2mm]
0, \hspace{1cm} {\rm otherwise}, \end{array}\right.
\end{align}
involving the discrete Heaviside step function defined on 
$\mathbb{Z}$:
\beq
\label{discrete-Heaviside}
\Theta(r) =  \left\{\begin{array}{cl} 1 ,\hspace{1cm} r \geq 0,  \\[2mm]
	0 , \hspace{1cm} r<0. \end{array}\right. 
\eeq
Note that for $n=0$ we have $\Theta(0,t,\Delta t_1) = \Theta(t)-\Theta(t-\Delta t_1)= \Theta(\Delta t_1-1-t)$.
Allowing for a random choice of the relocation node, we introduce the \emph{relocation matrix} $\mathbf{R}$ for which $R_{ij}$ is the probability that the walker being in node $i$ is relocated at node $j$. We consider the case in which $\mathbf{R}$ has identical rows, i.e.\ $R_{ij} =  R_{j}$, with row normalization $\sum_{j=1}^NR_j=1$.    
As a special case, $R_j = \delta_{rj}$ for all $j=1,\dots,N$, if the relocation happens on a single node $r$ almost surely. 
We observe the following feature of the relocation matrix $\mathbf{R}=[R_j]$:
\beq
\label{absorbing_feature}
[\mathbf{W} \cdot \mathbf{R}]_{ij} = \sum_{k=1}^N W_{ik} R_j = R_j = R_{ij}.
\eeq
Therefore $\mathbf{W}^t\cdot \mathbf{R} = \mathbf{R}$ and also $\mathbf{R}^t= \mathbf{R}$.
On the other hand, 
\beq
\label{R_W}
[\mathbf{R}\cdot \mathbf{W}]_{ij} = \sum_{k=1}^N R_k W_{kj} = W'_{j} \hspace{1cm} ( \neq [\mathbf{W} \cdot \mathbf{R}]_{ij} )
\eeq
and for $R_{ij}=\delta_{rj}$, (\ref{R_W}) yields $W'_{ij} = W_{rj}$.

With these ingredients, we are ready to establish the transition matrix (the propagator) $\mathbf{P}(t)$ of the walk 
with resets,
\beq
\label{walk_with_resets}
 \mathbf{P}(t) =   \mathbf{W}^t  \big\langle \Theta(\Delta t_1-1-t) \big\rangle + \mathbf{R} \cdot 
 \sum_{n=1}^{\infty} \big\langle \Theta(J_n,t,J_{n+1}) \mathbf{W}^{t-J_n} \big\rangle , \hspace{1cm} t \in \mathbb{N}_0,
\eeq
where
the so-called \emph{backward recurrence time} $B_{n,t} = t-J_n \geq 0$, i.e., the delay between the time $t$ and the last renewal
time $J_n$, comes into play.
The first term in (\ref{walk_with_resets}) corresponds to the walk before the first reset and  $\big\langle \Theta(\Delta t_1-1-t) \big\rangle = \sum_{r=t+1}^{\infty} \psi(r) =\Phi^{(0)}(t)$ is the probability that no reset occurs up to and including time $t$ (with the initial condition $\Phi^{(0)}(0)=1$ due to the fact that ${\cal N}(0)=0$). Recall also that the $i$-th row of $\mathbf{P}(t)$ represents the time-evolution of the occupation probabilities of the nodes if the walk starts at node $i$.

We refer to \cite{Barkai-etal-2023} for a thorough investigation of continuous-time 
resetting renewal processes, containing a discussion on the statistics of the backward recurrence time. See also \cite{GodrecheLuck2001} for occupation time statistics, again for continuous-time renewal processes.

We recall now some aspects of discrete-time renewal processes, where we focus on the statistics of backward recurrence times (consult also \cite{SRW2023,GSD-Squirrel-walk2023} for details).

Consider the probability $f(t,b,n)$ that the random variable $B_{n,t}=t-J_n$ equals $b$ and ${\cal N}(t)=n$
(we use $\delta_{i,j}=\delta_{ij}$ for the Kronecker symbol), that is
\beq
\label{f_t_B_n}
f(t,b,n) =  \big\langle \Theta(J_n,t,J_{n+1}) \delta_{t-J_n,b} \big\rangle ,\hspace{1cm} t, b, n \in \mathbb{N}_0.
\eeq
Note that for $n=0$ we have $f(t,b,0) = \Phi^{(0)}(t) \delta_{tb}$ (as $J_0=0$).
Useful for the following analysis is the discrete PDF $f(t,b)$ of the backward recurrence time $B_t$, independently on the choice on n, which we define from 
(\ref{f_t_B_n}) by
$$
f(t,b) = 
\mathbb{P}[B_t=b] = \sum_{n=0}^{\infty} f(t,b,n) = \sum_{n=0}^{\infty} \left\langle \Theta(J_n,t,J_{n+1}) \delta_{t-J_n,b} \right\rangle .
$$
The normalization of $f(t,b)$ with respect to $b$ follows from 
$\sum_{b=0}^{\infty} f(t,b) = \sum_{n=0}^{\infty} \Phi^{(n)}(t) =1 $  (see (\ref{state_probs})).
It is useful to take the double GF of $f(t,b,n)$ which yields 
\beq
\label{GF_uv}
\begin{array}{clr}
\ds {\bar f}(u,v,n)&  =\ds \sum_{t=0}^{\infty}\sum_{b=0}^{\infty} f(t,b,n)u^t v^b  , \hspace{0.5cm}
 |u| < 1, \hspace{0.25cm} |v| \leq 1 &\\ \\
& \ds  =\big\langle \sum_{t=J_n}^{J_{n+1}-1} u^t v^{t-J_n} \big\rangle = \big\langle u^{J_n} \big\rangle
\frac{1- \langle (uv)^{\Delta t_{n+1}} \rangle }{1-uv} = [{\bar \psi}(u)]^n \frac{1- {\bar \psi}(uv)}{1-uv} &
\end{array}
\eeq
where we have used that the times $\Delta t_j$ between the resets are IID together with (\ref{Gf_spi}). 
Setting $v=1$ gives the GF of the state probabilities \eqref{state_probs}. The double GF of $f(t,b)$ then yields
\beq
\label{GF_FTV}
{\bar f}(u,v)  = \frac{1}{1-{\bar \psi}(u)}   \frac{1- {\bar \psi}(uv)}{1-uv}.
\eeq
Further, we have ${\bar f}(0,v) = 1$, reflecting the initial condition of the state probabilities $\Phi^{(n)}(0) =\delta_{n0}$. Then, for subsequent use, it is useful to account for the
resetting rate
\beq
\label{resetting_rate}
{\cal R}(t) = \sum_{n=1}^{\infty} \Psi_n(t),  
\eeq
where
\beq
\Psi_n(t)= \sum_{r=0}^t \psi(r)\Psi_{n-1}(t-r), \qquad
\Psi_{0}(t) =\delta_{t0} , \qquad t= 1,2,\ldots
\eeq
Since ${\bar \Psi}_n(u)=[{\bar \psi}(u)]^n$, the  GF of the resetting rate is ${\bar {\cal R}}(u) = \frac{{\bar \psi}(u)}{1-{\bar \psi}(u)}$.
We can rewrite (\ref{GF_FTV}) as ${\bar f}(u,v) = (1+{\bar {\cal R}}(u)) \frac{1- {\bar \psi}(uv)}{1-uv}$
which, by full inversion, gives (see Appendix \ref{derivations_18_19} for a detailed derivation)
\beq
\label{ftB}
f(t,b) = \Phi^{(0)}(b) \left(\delta_{tb}+{\cal R}(t-b) \right) , \hspace{1cm} b, t = 0, 1,2, \ldots
\eeq
where ${\cal R}(t-b)=0$ for $b \geq t$ as ${\cal R}(\tau)$ (as well as $\psi(\tau)$) is null for $\tau < 1$ and
$ \Phi^{0}(b) =1-\sum_{r=1}^b\psi(r)$ is the persistence probability.
Hence, $f(t,b) = 0$ for $b >t$.
Inverting (\ref{GF_FTV}) with respect to $u$ leads to the representation
\beq
\label{Gf_B}
{\bar f}(t,v)
= \Phi^{(0)}(t) v^t + \sum_{b=0}^t{\cal R}(t-b) \Phi^{(0)}(b) v^b.
\eeq
Since ${\bar f}(u,v) = \frac{1- {\bar \psi}(uv)}{1-uv} + {\bar \psi}(u){\bar f}(u,v)$ one can see that (\ref{Gf_B})
fulfills the renewal equation
\beq
\label{renewal_f}
{\bar f}(t,v) = \Phi^{(0)}(t) v^t + \sum_{k=1}^t \psi(k) {\bar f}(t-k,v)
\eeq
where we considered $\psi(0)=0$.
A similar equation will also arise in the context of walks with resetting, which we will cope with in the following sections.
Moreover, inverting (\ref{GF_FTV}) with respect to $v$ gives ${u^b \Phi^{0}(b)}/({1-{\bar \psi}(u)})$ and therefore $f(t,b)$ takes the stationary distribution 
\beq
\label{barkai_agrement}
f(\infty,b) =  \lim_{u\to 1-}  u^b \Phi^{0}(b)  \frac{(1-u)}{1-{\bar \psi}(u)}  = \frac{\Phi^{0}(b)}{\langle \Delta t \rangle}
\eeq
if the mean resetting time is finite.
This relation is the discrete-time counterpart to the one reported in the literature \cite{Barkai-etal-2023,GodrecheLuck2001,Dynkin1955}. Lastly, from
\beq
\label{fB_limit}
 {\bar f}(\infty,v) 
 = \frac{1}{\langle \Delta t \rangle}\frac{1-{\bar \psi}(v)}{1-v} ,\hspace{1cm} |v| \leq 1,
\eeq
one can see that $f(\infty,b)$ is a proper PDF, as ${\bar f}(\infty,v) \to 1$ for $v\to 1$.

A natural question arises: what happens if
$\langle \Delta t\rangle$ is infinity? This is the class of random variables for which $\psi(t)$ is fat-tailed and $f(\infty,b)\to 0+$ for every finite $b$ (see e.g.\ the Sibuya case (\ref{Dynkin}), in Appendix \ref{standard_Sib}, where ${\bar \psi}_{\alpha}(u)=1-(1-u)^{\alpha}$), thus $\Phi^{(0)}(b)$ is not a normalizable function.
Indeed, $f(\infty,b)$ then is a defective PDF with respect to $b$. 
Consult \cite{Dono_Mi_Po_Ria2024} for the related theory, and see also \cite{Gelfand-Shilov1968}.
The backward recurrence time $B$ is concentrated at infinity. 
To see this, consider (\ref{GF_FTV}) where ${\bar f}(u,1)=1/(1-u)$ shows that $f(t,b)$ is a properly normalized (non-defective) PDF of $b$ for any finite observation time, where 
\beq
\label{inf_time}
{\bar f}(\infty,1) := \lim_{t\to \infty} \sum_{b=0}^t f(t,b) = 1 
\eeq
corresponds to the physical situation that we extend our observation time larger and larger. 
In the Sibuya case the infinite time limit (\ref{inf_time}) is retrieved from
\beq
\label{definilimit}
{\bar f}_{\alpha}(\infty,v) = 
\lim_{u\to 1-} (1-u) f_{\alpha}(u,v) = (1-u)^{1-\alpha} (1-uv)^{\alpha-1}\bigg|_{u=1} = \left\{\begin{array}{cl} 0 , & |v|< 1 \\ \\
1 , & v=1 
\end{array}\right.  \hspace{0.5cm} \alpha \in (0,1) .
\eeq
Evoking Tauberian arguments,
one can see that for $|v|<1$ one has ${\bar f}_{\alpha}(u,v) \sim (1-v)^{\alpha-1} (1-u)^{-\alpha} $ as $u\to 1-$. The zero value is approached by a slow power-law 
${\bar f}_{\alpha}(t,v) \sim (1-v)^{\alpha-1} t^{\alpha-1}/\Gamma(\alpha)$, as $t\to \infty$.
We will see a little later that (\ref{definilimit}) is indeed crucial
for the non-existence of a NESS for this RRP class.
We refer also to the extensive discussions in \cite{Barkai-etal-2023,WangSchulzBarkai2028}.

\medskip

Now, we apply the above results to obtain the following GF of the propagator (\ref{walk_with_resets}) written in matrix form
\beq
\label{Gf_transition matrix} 
{\bar {\bf P}}(u) = \left({\mathbf 1}  +\frac{{\bar \psi}(u)}{1-{\bar \psi}(u)}\mathbf{R}\right)\cdot \frac{{\mathbf 1} - {\bar \psi}(u\mathbf{W})}{{\mathbf 1} - u\mathbf{W}}   , \hspace{1cm} |u|< 1
\eeq
where $ \mathbf{1}$ stands for the unity matrix. Observe that $\frac{{\mathbf 1} - {\bar \psi}(u\mathbf{W})}{{\mathbf 1} - u\mathbf{W}}$ is the GF of $\mathbf{W}^t\Phi^{(0)}(t)$ (first term in (\ref{walk_with_resets})). 
Let us check the non-Markovianity of the walk. By inverting (\ref{Gf_transition matrix}) with respect to $u$ yields
\beq
\label{propagator_time}
\mathbf{P}(t) = \Phi^{0}(t)  \left(\mathbf{1}- \mathbf{R}\right) \cdot \mathbf{W}^t + \mathbf{R} \cdot {\bar f}(t,\mathbf{W}) .
\eeq
Use (\ref{renewal_f}) together with $\mathbf{R}\cdot \mathbf{P}(t) = \mathbf{R} \cdot {\bar f}(t,\mathbf{W})$ to arrive at the renewal equation for $\mathbf{P}(t)$:
\beq
\label{renewal_structure}
\mathbf{P}(t) =  \Phi^{(0)}(t) \mathbf{W}^t + \mathbf{R}\cdot\sum_{k=1}^t \psi(k) \mathbf{P}(t-k)
\eeq
This relation is a ``network discrete-time version'' of the renewal equation of random motions under resetting reported in the literature
\cite{Evans-et-al2011,Evans-etal2020,Sandev-et-al2022,Sandev-Mendez-et-al2023, GamlmesCamosMendez2020} and generalizes the one given in \cite{ChenYe2022} to a random choice of r-nodes.
Particularly useful is the canonical representation of
(\ref{Gf_transition matrix})\footnote{Where we use $\mathbf{R} \cdot |\phi_1\rangle \langle {\bar \phi}_1| =  |\phi_1\rangle \langle {\bar \phi}_1|$.}
\begin{multline}
\label{GF_state_resets}
{\bar P}_{ij}(u)  = {\bar P}_{ij}(u) =
\frac{1}{(1-u)} \frac{K_j}{\sum_{r=1}^NK_r} \\+  \sum_{m=2}^N \frac{1-{\bar \psi}(u\lambda_m)}{1-\lambda_m u}\left( 
\langle i|\phi_m\rangle \langle {\bar \phi}_m|j\rangle + \frac{{\bar \psi}(u)}{1-{\bar\psi}(u)} \sum_{r=1}^N R_r\langle r |\phi_m\rangle \langle {\bar \phi}_m|j\rangle\right) .
\end{multline}
Inverting this GF yields
\begin{multline}
\label{switch-back}
P_{ij}(t) = \frac{K_j}{\sum_{r=1}^NK_r} \\+  \sum_{m=2}^N \left( \Phi^{(0)}(t) \lambda_m^t \,\langle i|\phi_m\rangle \langle {\bar \phi}_m|j\rangle +
\left({\bar f}(t,\lambda_m) - \Phi^{(0)}(t) \, \lambda_m^t  \right) \sum_{r=1}^N R_r \langle r |\phi_m\rangle \langle {\bar \phi}_m|j\rangle \right).
\end{multline}
Accounting for (\ref{fB_limit}) takes us to the infinite time limit (NESS):
\beq
\begin{aligned}
 \ds & \mathbf{P}(\infty)   = \ds \mathbf{R} \cdot {\bar f}(\infty,\mathbf{W}) & \\
\ds & P_{ij}(\infty) = \ds P_{j}(R_1, \ldots, R_N; \infty) =
  \frac{K_j}{\sum_{r=1}^NK_r} + \frac{1}{\langle \Delta t \rangle}  \sum_{m=2}^N \frac{1-{\bar \psi}(\lambda_m)}{1-\lambda_m} \sum_{r=1}^N R_r \langle r|\phi_m\rangle \langle {\bar \phi}_m| j\rangle,
  \end{aligned}
\label{NESS}
\eeq
which does not depend on the initial node $i$, that is this matrix has identical 
rows. Further, it contains the averaging of the relocation node $r$ with respect to the probabilities $R_r$ and preserves row normalization. The NESS propagator fulfills $ \mathbf{P}(\infty) = \mathbf{R} \cdot \mathbf{P}(\infty)$.
Both (\ref{NESS}) and $\mathbf{R}$ have rank one with a single eigenvalue equal to one 
(as $\mathbf{P}(\infty)\cdot |\phi_1\rangle = \mathbf{R}\cdot |\phi_1\rangle = |\phi_1\rangle$), and $N-1$ zero eigenvalues.
The existence of a NESS (i.e.\ with non-vanishing second term in (\ref{NESS})) requires a finite mean $\langle \Delta t \rangle$ of the time intervals between consecutive resets. 
This was also found in earlier works for continuous space-time walks with resetting
\cite{Sandev-et-al2022,Barkai-etal-2023} and remains true in infinite networks, where $ K_j/(\sum_{r=1}^NK_r) \to 0$.
For the cases of distributions of the time intervals between consecutive resets with diverging mean, the second term in (\ref{NESS}) is suppressed (see (\ref{definilimit})), thus the stationary (equilibrium) distribution of the Markovian walk for infinite times is taken.

Worthy of mention is that if we choose as relocation probabilities $R_j=K_j/(\sum_{r=1}^NK_r)$ (the stationary distribution of the walk without resetting) then
the second term in (\ref{NESS}) is wiped out and no NESS can exist due to $\sum_{r=1}^NK_r \langle r| \phi_m\rangle =\langle {\bar \phi}_1|\phi_m\rangle  =0$, $m=2,\ldots ,N$. Thus, $P_{j}(\infty)$ again coincides with the (equilibrium) stationary distribution of the Markovian walk without resetting.
Keep in mind that the results derived so far hold for any (Markovian and non-Markovian) RRP. We consider next a few scenarios of Markovian resetting.
\subsection{Geometric resetting time intervals}
\label{Markovian_reseting}
The Markovian resetting to a single \cite{Riascos_boyer_stat_res2020} and to two resetting nodes \cite{Riascos-Boyer-etal2021,Julian_Salgado_etal2024} was explored recently. 
Here we consider Markovian resetting to randomly chosen nodes defined by the above introduced relocation matrix $\mathbf{R} = [R_j]$.
For a Markovian resetting, ${\cal N}(t)$ is a Bernoulli process (the discrete-time counterpart of a Poisson process). Then, the resetting PDF (\ref{discret-PFD})
is a geometric distribution $\psi(t) = pq^{t-1}$, $t\geq 1$, where $p \in (0,1)$ denotes the probability of a Bernoulli reset 
(and equals the constant resetting rate 
${\cal R}= p$, $t\geq 1$, defined by (\ref{resetting_rate}), see e.g.\, \cite{SRW2023}). We denote by $q=1-p$ the complementary probability.
The geometric distribution has GF ${\bar \psi}(u)= \frac{pu}{1-qu}$. 
That fact that geometrically distributed resetting times entails the Markov property can be seen by considering
the GF of the discrete-time memory kernel 
$ {\bar K}(u) = (1-u)\frac{{\bar \psi}(u)}{1-{\bar \psi}(u)} = p u $
and hence $K(t) = p \delta_{t1}$, which is null for $t>1$. 
The GF (\ref{Gf_B}) has then the form
\beq
\label{markov-distributions}
{\bar f}(t,\lambda) = \frac{p}{1-q\lambda} + \frac{q(1-\lambda)}{1-q\lambda}(q\lambda)^t , \hspace{1cm} \Phi^{(0)}(t) = q^t , \hspace{1cm} |\lambda| \leq 1
\eeq
with the stationary value ${\bar f}(\infty,\lambda) = \frac{p}{1-q\lambda}$ which we can invert with respect to $\lambda$ to recover 
$f(\infty,b) = pq^b$ supported on $\mathbb{N}_0$, in accordance with 
(\ref{barkai_agrement}).
Then we have for (\ref{Gf_transition matrix}),
\beq
\label{Bernoulli_GF_prop}
{\bar {\bf P}}(u) = \left(\mathbf{1} + \frac{pu}{1-u} \mathbf{R}\right) \cdot \frac{1}{\mathbf{1}-qu\mathbf{W}}
\eeq
which takes us to (see also (\ref{switch-back}) and (\ref{markov-distributions})) 
\beq
\label{switch-back_ber}
P_{ij}(t) = \frac{K_j}{\sum_{r=1}^NK_r} +  
\sum_{m=2}^N \left\{(q\lambda_m)^t \,\langle i|\phi_m\rangle \langle {\bar \phi}_m|j\rangle + \frac{p(1-(\lambda_m q)^t)}{1-q\lambda_m}
\sum_{r=1}^N R_r \langle r |\phi_m\rangle \langle {\bar \phi}_m|j\rangle \right\}.
\eeq 
Letting $t\to \infty$ we obtain for the NESS,
\beq
\label{NESS_Bernoulli}
\begin{array}{llr} 
\ds \mathbf{P}(\infty)   =    p \mathbf{R}  \cdot \frac{1}{\mathbf{1}-q\mathbf{W}} & & \\[2mm]
\ds P_{ij}(\infty) = P_{j}(p; R_1, \ldots, R_N; \infty)  = \ds P_{j}(\infty) = \frac{K_j}{\sum_{r=1}^NK_r} + p \sum_{m=2}^N 
\sum_{r=1}^N R_r \frac{\langle r| \phi_m\rangle \langle {\bar \phi}_m|j\rangle}{1-\lambda_m q}, & &
\end{array}
\eeq
which is independent of the starting node $i$. Further, $\langle \Delta t \rangle =1/p$ and $\frac{1}{1-\lambda_m q} = \frac{1-{\bar \psi}(\lambda_m)}{1-\lambda_m}$, see (\ref{NESS}). The first term corresponds to the stationary distribution of $\mathbf{W}^t$ (as $t\to \infty$), whereas the second one is related to the non-equilibrium resetting behavior. Note that, 
for the numerical computation of the NESS it is sufficient to invert $\mathbf{1}-q\mathbf{W}$ in the matrix representation (\ref{NESS_Bernoulli}) (having eigenvalues $q|\lambda_m|< 1$).  The inversion can be performed for $p \in (0,1]$ without the need of determining of the spectral quantities of $\mathbf{W}$.
Expression (\ref{NESS_Bernoulli}) generalizes the result of the NESS obtained in 
\cite{Riascos_boyer_stat_res2020,Riascos-Boyer-etal2021,ChenYe2022,HuangChen2021} to an arbitrary set of randomly chosen nodes.

Instructive is the deterministic limit $p=1$ for one relocation node, say $r_0$ ($R_j=\delta_{j,r_0}$) where the walker is constantly relocated at $r_0$, thus remaining trapped there forever. The transition matrix becomes stationary for $t\geq 1$ with $P_{ij}(t) = P_j(1;R_1,\ldots,R_N;\infty)= \delta_{j,r_0}$.
One observes that deterministic (periodic) resets with short periodicity may prevent the walker to explore the complete network and hence ergodicity of the walk will break down. We come back to this issue later on.

Let us explore whether the considered Markovian walk with geometric resetting is a Markov chain. To this end, we write its matrix form
\beq
\label{matrix_transition matrix_reset}
\mathbf{P}(t) = \mathbf{W}^t q^t +p \mathbf{R}\cdot \frac{\mathbf{1}-q^t\mathbf{W}^t}{\mathbf{1} - q\mathbf{W}} ,\hspace{1cm} t=0,1,2,\ldots 
\eeq
Then, we have $\mathbf{P}(1) = q\mathbf{W}+p\mathbf{R}$ which is indeed the one-step transition matrix of the model in \cite{Riascos_boyer_stat_res2020}). Note that
\beq
\label{markov_check}
\begin{aligned}
\mathbf{P}(t)\cdot \mathbf{P}(1) &= 
\left(\mathbf{W}^t q^t +p \mathbf{R}\cdot  \frac{\mathbf{1}-q^t\mathbf{W}^t}{\mathbf{1} - q\mathbf{W}}\right)  
\cdot \left( q\mathbf{W}+p\mathbf{R}\right) \\[2mm]
&= \mathbf{W}^{t+1} q^{t+1} +p \mathbf{R}\cdot \frac{\mathbf{1}-q^{t+1}\mathbf{W}^{t+1}}{\mathbf{1} - q\mathbf{W}} =\mathbf{P}(t+1), 
\end{aligned}
\eeq
where we used ${\bar g}(u\mathbf{W})\cdot\mathbf{R} = \mathbf{R}\,{\bar g}(u)$.
Expression (\ref{markov_check}) shows that (\ref{matrix_transition matrix_reset}) is indeed a Markov chain and can be rewritten as
\beq
\label{equival}
\mathbf{P}(t)=\left( q\mathbf{W}+p\mathbf{R}\right)^t .
\eeq
Plainly, it remains Markovian in the deterministic limit $p=1$ where $P(t)= \mathbf{R}^t =\mathbf{R}$ for $t\geq 1$ and $\mathbf{P}(0)=\mathbf{1}$.

\subsubsection{First passage quantities for geometric resetting}
\label{first_passage}
We derive the matrix of the first passage probabilities to the nodes. 
Although the MFPT in networks and related aspects were analyzed thoroughly in the mentioned recent works \cite{Riascos_boyer_stat_res2020,Riascos-Boyer-etal2021}, the present section aims at complementing their results by elaborating on some aspects which were not considered so far. 

Let $\mathbf{F}(t) = [F_{ij}(t)]$ denote the probabilities that a walker starting from $i$ at $t=0$ reaches node $j$ at time $t$ for the first time. By conditioning recurring to the Markov property we can write \cite{NohRieger2004} 
\beq
\label{first-passage_proba}
P_{ij}(t) = \delta_{t0}\delta_{ij}+ \sum_{r=0}^t P_{jj}(t-r) F_{ij}(r)
\eeq
where $P_{ij}(t)$ is given by the equivalent expressions (\ref{matrix_transition matrix_reset}) and (\ref{equival}). 
In (\ref{first-passage_proba}), we have $F_{ij}(0)=0$ and $ P_{jj}(t-r)$ is the probability of returning at node $j$ in $t-r$ steps. The GF of (\ref{first-passage_proba}) with respect to $t$ is
${\bar P}_{ij}(u)= \delta_{ij} + {\bar P}_{jj}(u) {\bar F}_{ij}(u)$, thus,
letting $u=e^{-s}$, and considering that $s \to 0$,
\beq
\label{well_known}
{\bar F}_{ij}(u) = \frac{{\bar P}_{ij}(u)-  \delta_{ij}}{{\bar P}_{jj}(u)} = 
\frac{P_{j}(\infty) + s ({\bar r}_{ij}(1)-\delta_{ij}) }{P_{j}(\infty) + s \, {\bar r}_{jj}(1) } + o(s),
\eeq
with ${\bar P}_{ij}(u) = P_{j}(\infty)/(1-u) + {\bar r}_{ij}(u) = P_{j}(\infty)/s + {\bar r}_{ij}(1) +O(s)$, $s\to 0$ and 
where $P_j(\infty)=P_{j}(p;R_1, \ldots, R_N; \infty)$ is the NESS (\ref{NESS_Bernoulli}). In this relation we introduced the GF ${\bar r}_{ij}(u)$ of the decaying part in $r_{ij}(t)= P_{ij}(t)-P_j(\infty)$ which can be extracted from 
(\ref{switch-back_ber}) and gives (see Appendix \ref{appendix_num_eval}, Eq.\ (\ref{Pinftyuone}), for the matrix representation)
\beq
\begin{aligned}
{\bar r}_{ij}(u)\bigg|_{u=1} &= \sum_{t=0}^{\infty} (P_{ij}(t) -P_j(\infty))\\
&= \sum_{m=2}^N \left\{\frac{1}{1-\lambda_m q}
\,\langle i|\phi_m\rangle \langle {\bar \phi}_m|j\rangle - \frac{p}{(1-\lambda_m q)^2}
\sum_{r=1}^N R_r \langle r |\phi_m\rangle \langle {\bar \phi}_m|j\rangle \right\} .
\end{aligned}
\label{relaxing_parts}
\eeq
The mean first passage time (MFPT) of the walker to reach node $j$ (with starting node $i$) is
\beq
\begin{aligned}
\ds \big\langle T_{ij}& (p; R_1,\ldots, R_N)\big\rangle = \ds -\frac{d}{ds} {\bar F}_{ij}(e^{-s})\bigg|_{s=0} = \frac{1}{P_{j}(\infty)} \left(
{\bar r}_{jj}(1) -{\bar r}_{ij}(1) +\delta_{ij}\right) \\[2mm]
   & = \ds \frac{1}{P_j(p; R_1, \ldots, R_N;\infty)}\left(\delta_{ij} + \sum_{m=2}^{N} \frac{1}{1-q\lambda_m}[ \,\langle j|\phi_m\rangle \langle {\bar \phi}_m|j\rangle - \,\langle i|\phi_m\rangle \langle {\bar \phi}_m|j\rangle ] \right) . 
\end{aligned}
\label{MFPT}
\eeq
The term in the second line of (\ref{relaxing_parts}), which is independent on the starting node $i$, cancels out 
in ${\bar r}_{jj}-{\bar r}_{ij}$ of (\ref{MFPT}). The MFPT depends on the relocation matrix via the NESS given in (\ref{NESS_Bernoulli}) and is an asymmetric matrix, a feature
which is well-known for Markovian walks without resetting \cite{NohRieger2004}. For $j=i$ we have 
$\langle T_{ii} \rangle =1/P_{i}(p; R_1,\ldots, R_N;\infty)$ as the mean recurrence time to the starting node $i$, in accordance with Kac's Lemma \cite{Kac1947,Masuda-et-al2017}. For the purpose of numerical evaluations, it is useful to notice that (\ref{MFPT}) can be conveniently expressed by the elements of the matrix 
\beq
\label{S-matrix}
\mathbf{S}(p) = [1-q\mathbf{W}]^{-1} -\frac{1}{p}|\phi_1\rangle\langle{\bar \phi}_1|  ,\hspace{1cm} p \in (0,1]
\eeq
where $S_{ij} ={\bar r}_{ij}(u)\bigg|_{u=1}$. 
Consult Appendix \ref{appendix_num_eval} for some brief derivations.
We can hence express the MFPT 
(\ref{MFPT}) by using (\ref{S-matrix}) as
\beq
\label{represntation_Tij}
\big\langle T_{i j}(p; R_1,\ldots, R_N)\big\rangle = \frac{1}{P_{j}(p; R_1, \ldots, R_N;\infty)}\left(\delta_{ij} + S_{jj}(p) -S_{ij}(p)\right)  ,\hspace{1cm} p \in (0,1] .
 \eeq
Expressions (\ref{MFPT}) and (\ref{represntation_Tij}) generalize the MFPT obtained in the references \cite{Riascos_boyer_stat_res2020,Riascos-Boyer-etal2021,ChenYe2022}
and are, to our knowledge, not reported in the literature so far. 
In order to have a global measure of the time needed by a searcher to find a target in the network, we consider the Kemeny constant $\mathcal{K}$ under resetting (global MFPT) as
\beq
\label{Kemeny_def}
{\cal K}(p) +1 = \big\langle T \big\rangle = \sum_{j=1}^{\infty} \big\langle T_{i j}(p; R_1,\ldots, R_N)\big\rangle P_{j}(p; R_1, \ldots, R_N;\infty)
\eeq
where we consider the general definition given by Kemeny and Snell \cite{KemenySnell1960} (shifted by $1$ to account for visits of nodes different from the starting node $i$ only). 
The Kemeny constant has the interpretation of the expected time a searcher needs to reach a target node $j$ which is randomly chosen with (NESS) probability $P_j(p;R_1,\ldots,R_N,;\infty)$. The Kemeny constant thus contains global information about the efficiency of a searcher in the network. 
This straightforwardly leads to
\beq
\label{Kemeny_result}
{\cal K}(p) = \mathrm{tr}\{ \mathbf{S}(p) \} = \sum_{m=2}^N\frac{1}{1-\lambda_m q}
\eeq
where we used that $\sum_{j=1}^N S_{ij} = \langle i| \mathbf{S}|\phi_1\rangle = 0$ as $\langle i|\phi_1\rangle$ is constant. The 
Kemeny constant is independent of the starting node $i$, which generally holds true for Markovian walks, as shown in the 1960 seminal paper of Kemeny and Snell. For $p=0$, formula (\ref{Kemeny_result}) recovers the Kemeny constant of the Markovian walk without resetting \cite{NohRieger2004}. 
Interestingly enough, the Kemeny constant is also independent of the choice of the resetting matrix $[R_j]$ and depends only on the Bernoulli probability $p$ and the $N-1$ eigenvalues 
$|\lambda_m|<1$ of $\mathbf{W}$.

A further meaningful quantity is the mean relaxation time of a node $i$ \cite{NohRieger2004} (see
(\ref{relaxing_parts}) and Appendix \ref{appendix_num_eval}):
\beq
\label{mean_relax}
{\cal T}_i(p;R_1,\ldots,R_N) = 
\sum_{t=0}^{\infty} (P_{ii}(t)- P_{i}(\infty)) =  \left[ (\mathbf{1}-q\mathbf{W}-p\mathbf{R}) \, \cdot\, (\mathbf{1}-q\mathbf{W})^{-2} 
 \right]_{ii} , \hspace{0.5cm} p \in (0,1] .
\eeq
From this expression we define the global mean relaxation time of the nodes,
\beq
\label{Kemeny_res}
\begin{array}{clr}
\ds {\cal T}(p; R_1,\ldots, R_N) & = \ds \frac{1}{N}\sum_{i=1}^N{\cal T}_i(p;R_1,\ldots,R_N) = \frac{1}{N} \mathrm{tr}\{\left[ (\mathbf{1}-q\mathbf{W}-p\mathbf{R}) \, \cdot\, (\mathbf{1}-q\mathbf{W})^{-2}) \right] \}  & \\ \\
 & = \ds  \frac{1}{N} \sum_{m=2}^N \left\{\frac{1}{1-\lambda_m q}  - \frac{p}{(1-\lambda_m q)^2}
 \sum_{r=1}^N  R_r \langle r |\phi_m\rangle \sum_{i=1}^N \langle {\bar \phi}_m|i\rangle \right\}     & \\ \\
 & = \ds  \frac{1}{N} \sum_{m=2}^N \frac{1}{1-\lambda_m q}  = \frac{1}{N} {\cal K}(p) 
 \end{array}
\eeq
as for $m=2,\ldots,N$ we have $\sum_{i=1}^N \langle {\bar \phi}_m|i\rangle = \langle {\bar \phi}_m|\phi_1\rangle =0 $ ($\langle i|\phi_1\rangle$ is independent of $i$). Note that the global time  ${\cal T}(p; R_1,\ldots, R_N)=  {\cal T}(p)$ is independent of the resetting matrix.
Numerical evaluations of the quantities (\ref{represntation_Tij}) -- (\ref{Kemeny_res}) do not require the determination of the eigenvalues and eigenvectors of $\mathbf{W}$. 
For these evaluations it is sufficient to invert the matrix $\mathbf{1}-q\mathbf{W}$ numerically. The second term in $\mathbf{S}$ is proportional to
the stationary distribution $K_j/(\sum_{s=1}^N K_s)$. After the invertion, a numerical evaluation is performed by taking the traces of $\mathbf{S}$ and of the matrix 
in (\ref{Kemeny_res}).
The case without resetting where $\mathbf{1}-\mathbf{W}$ is not invertible (as $\lambda_1=1$) can be approached as the limit $p \to 0+$, where the well-known classical expression for Markovian walks without resetting is eventually recovered \cite{MiRia2017,NohRieger2004}.

Now, we introduce the following measure for the efficiency of a search strategy,
\beq
\label{searcher_efficieny}
E(p)= \frac{N}{{\cal K}(p)}  
\eeq
where $E(p)\big|_{p=1-} = N/(N-1) \approx 1$ ($N\gg 1$) which is very close to the maximum value $E_{CC}(0)= N^2/(N-1)^2$ in completely connected networks, which we prove hereafter.
Let us analyze more closely the Kemeny constant:
\beq
\label{kem_expan}
{\cal K}(p) = \sum_{m=2}^N \frac{1}{1 +(p-1)\lambda_m} 
\eeq
Observe that its second derivative
\beq
\label{kemeny_ineq}
 {\cal K}^{''}(p) = 2\sum_{m=2}^N \frac{\lambda_m^2}{(1 +(p-1)\lambda_m)^3}  > 0
\eeq
is a strictly positive function of $p$.
Therefore, 
$$  {\cal K}^{'}(0)   <  {\cal K}^{'}(p) = - \sum_{m=2}^N \frac{\lambda_m}{(1 +(p-1)\lambda_m)^2}  < {\cal K}^{'}(1) =  -\sum_{m=2}^N \lambda_m = 1 .$$
Notice that, since $W_{ii} = 0$ per construction, we have $\sum_{m=2}^N \lambda_m = \mathrm{tr}\{\mathbf{W}\} -1 = -1$ ($\lambda_1=1$). Hence, ${\cal K}(p)$ is a convex function of $p$. So if ${\cal K}^{'}(0) < 0$ then there exist an optimal $p=p_*$ such that ${\cal K}^{'}(p_{*})= 0$ where ${\cal K}(p_{*})$ is minimal.
As examples, in Fig.\ \ref{Kemeny_BA_WS}, we depict the Kemeny constant as a function of the resetting rate $p$ for the Barab\'asi-Albert and Watts-Strogatz graphs considered in the following.
\paragraph{Completely connected network}
As an instructive case to understand the role of connectivity, consider the Kemeny constant in a completely connected (CC) network, that is each node is connected to any other node except to itself. All nodes have degree $K_i^{CC}=N-1$ and the adjacency matrix has the entries $A_{ij}^{CC}= 1-\delta_{ij}$. CC networks have transition matrix \cite{MiRia2017} 
\beq
\label{transmat_cc}
\mathbf{W}^{CC} = [(1-\delta_{ij})/(N-1)] = \frac{N}{N-1} |\phi_1\rangle\langle {\bar \phi_1}| -\frac{1}{N-1}\mathbf{1} = |\phi_1\rangle\langle {\bar \phi_1}| -\frac{1}{N-1}\sum_{m=2}^N |\phi_m\rangle\langle {\bar \phi_m}|
\eeq
where the stationary state is $ \langle i|\phi_1\rangle \langle {\bar \phi_1}| j\rangle =1/N$, $i,j=1,\ldots,N$. Thus, we have $\lambda_m^{CC} =-1/(N-1)$, $m=2,\ldots,N$. The Kemeny constant can hence be explicitly written as
\beq
\label{kemen_CC}
{\cal K}_{CC}(p) = \frac{(N-1)^2}{N-p}
\eeq
which has a very weak dependence on $p$, and it is monotonically increasing with $p$, from ${\cal K}_{CC}(0)=  \frac{(N-1)^2}{N}$ to 
${\cal K}_{CC}(1)=N-1$ with ${\cal K}_{CC}(p) \approx N$ ($N\gg 1$).
We infer that for CC networks, resetting is not advantageous as
$E_{CC}(p) = N(N-p)/(N-1)^2 \approx 1$. 
Hence, in CC networks, there is no optimal resetting rate $p_*$.
\\ \\
Consider now the deterministic limit $p=1$ 
where the walker is relocated at each integer time instant $t\geq 1$ to a single relocation node, say $r$, with
$R_{j}=\delta_{jr}$. Then the walker remains trapped for $t\geq 1$ on the relocation node $r$ which is reflected by the NESS $P_j(\delta_{rj},\infty) = R_j =\delta_{rj}$. Thus, we get
\beq
\label{important_case}
 \langle T_{ij}(1;\delta_{r1},\ldots,\delta_{rN})\rangle  =\frac{1}{P_j(1;\delta_{r1},\ldots,\delta_{rN};\infty)} =\frac{1}{\delta_{rj}} =\left\{ \begin{array}{cl} \ds \infty , & \ds j \neq r, \\[2mm]
\ds    1 ,   &\ds  j=r. 
\end{array}\right. 
\eeq
The MFPT is exactly one only if the target node coincides with the resetting node. Otherwise, the target node can never be reached and the MFPT is infinity. 
In general, let ${\cal L}_R$ be the set of resetting nodes (r-nodes) to which the walker is allowed to be relocated with non-zero probabilities $R_k > 0$ if $k \in {\cal L}_R$ 
and $R_k=0$ for $k \notin {\cal L}_R$. Then, a target $j$ can only be found if $j\in {\cal L}_R$ and the MFPT is
\beq
\label{important_case_set_tarket}
 \langle T_{ij}(1;R_1,\ldots,R_N)\rangle  = \frac{1}{R_j} , \hspace{1cm} j \in {\cal L}_R
\eeq
as $P_j(1;R_1,\ldots,R_N;\infty)= R_j$  and $\langle T_{ij}(1;R_1,\ldots,R_N)\rangle = \infty$ if $j \notin {\cal L}_R$.
We intuitively infer that, in the deterministic limit,
the ergodicity of the walk breaks down if ${\cal L}_R$ 
does not comprise the entire network. We will confirm this assertion later on.
In the following two sections we explore the efficiency of a random searcher under Markovian resetting in realizations of random graphs of the Barabási-Albert (BA) and Watts-Strogatz (WS) type.
\paragraph{Barabási-Albert graph} 
We consider a random searcher on realizations of Barabási-Albert random 
graphs (BA graph in the following).
The BA graph is generated by a preferential attachment mechanism for the newly added nodes. We used in all examples the PYTHON NetworkX library.
To generate a BA graph, one starts with $m_0$ nodes. In each step a new node is added and connected with $m \leq m_0$ 
existing nodes which are selected with probability proportional to their degree. Thus, a newly added node is preferentially 
attached to a node with higher degree.
The quantity $m$ is referred to as the attachment parameter. In this growth process, a scale-free network with
an asymptotic power-law degree distribution, and the small world property emerges
\cite{Barabasi2016,Barabasi1999}.
\begin{figure}[t!]
\centerline{
\includegraphics[width=0.95\textwidth]{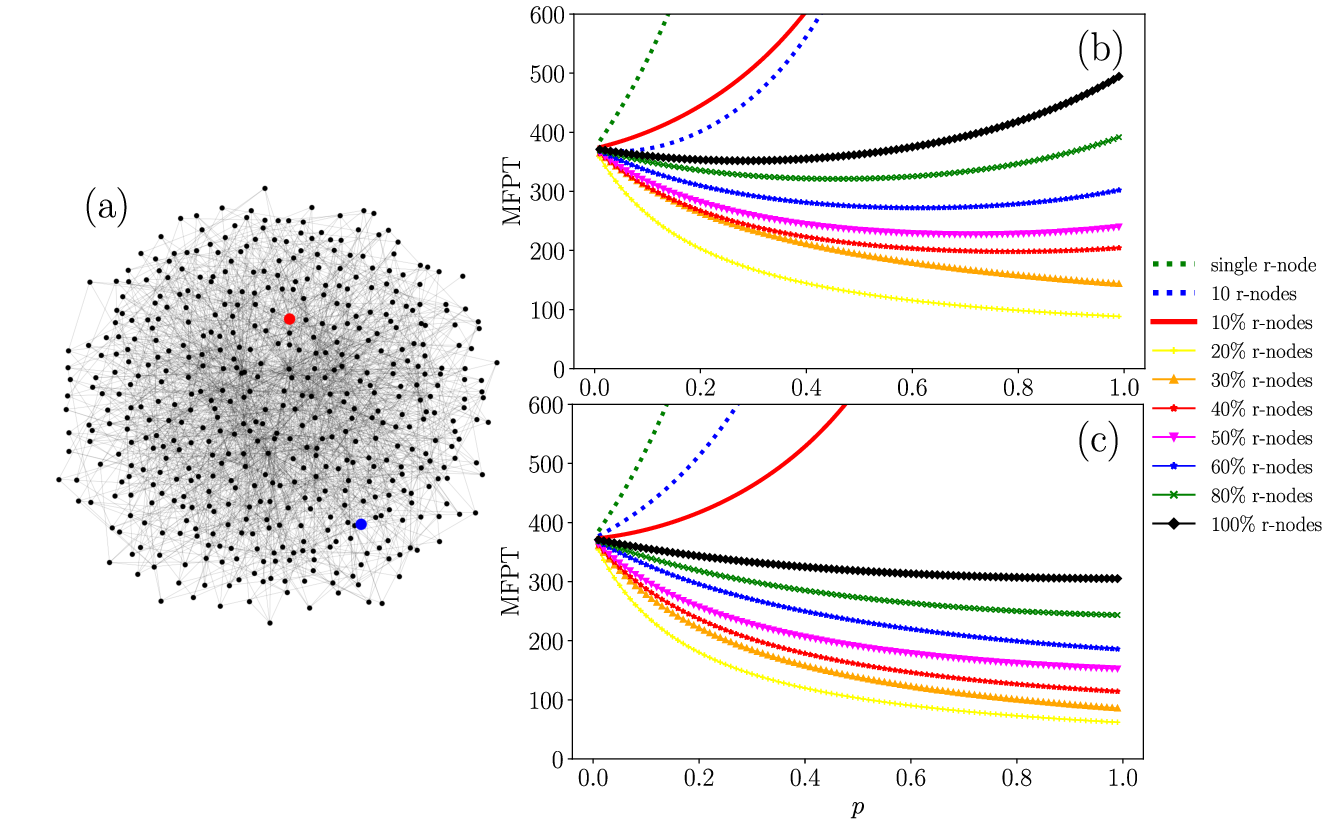}
}
\centerline{
\includegraphics[width=0.5\textwidth]{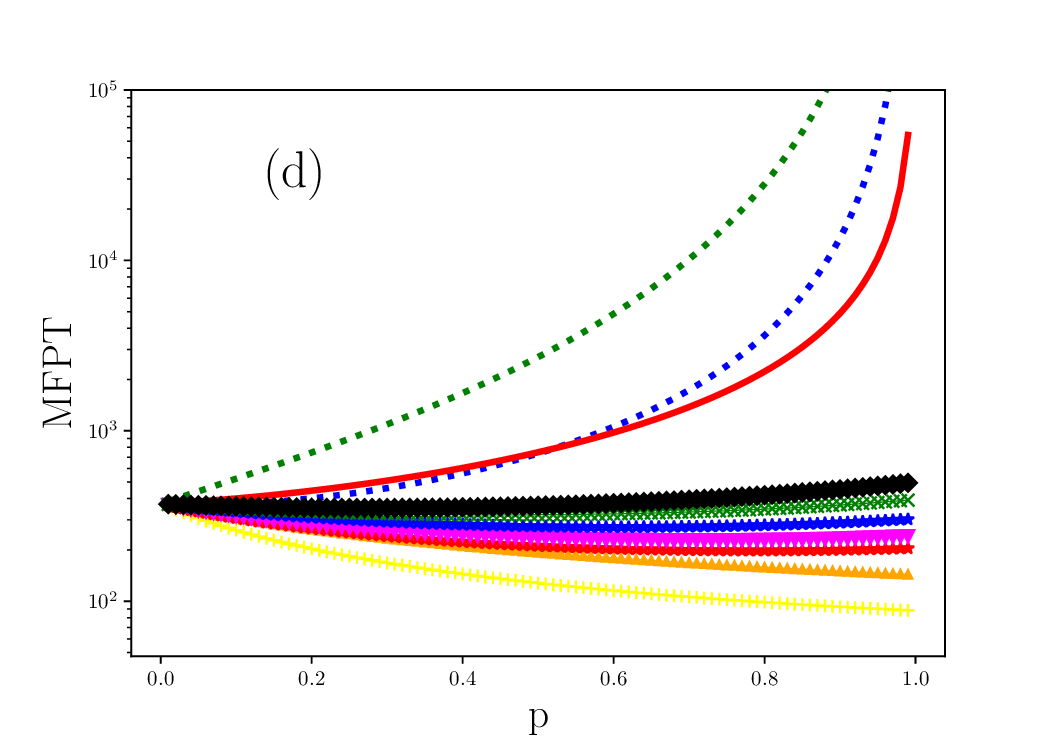}
\includegraphics[width=0.5\textwidth]{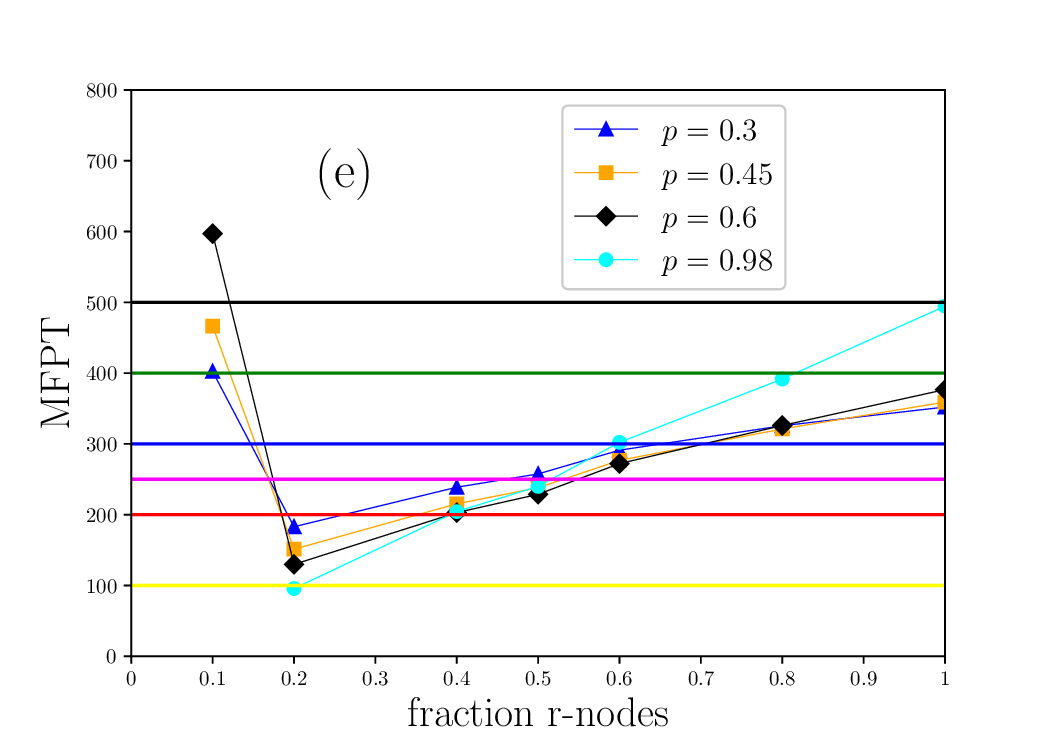}
}
\vspace{-2mm}
\caption{(a) Random search dynamics with resets on a realization of the  Barabási-Albert random graphs with $N=500$ nodes and attachment parameter $m=4$. 
Frames (b,c) show the MFPT $\langle T_{ij} \rangle$ of Eq.\ (\ref{represntation_Tij}) as a function of the Bernoulli probability $p$ 
for different relocation scenarios of the searcher. The sets of r-nodes to which resetting is allowed are
randomly selected by an independent Bernoulli trial for each node, generating homogeneously distributed fractions of r-node populations. 
We plot the MFPT for various fractions ($10\%$--$100\%$) of the r-node populations.
The starting node ($i=100$) is marked in blue, the target node ($j=200$) in red.
Frame (b): MFPT versus $p$ for uniform reset ($R_r=R$) to the r-nodes). Frame (c): MFPT versus $p$ for preferential resets to the r-nodes depending on the node degree ($R_r \propto K_r$). Frame (d): MFPT of (b) with a larger window. frame (e): MFPT as a function of the r-node fraction for some values of $p$ (same setting of (b)). }
\label{BA_network}
\end{figure}
In Fig.\ \ref{BA_network}, we depict the MFPT (\ref{represntation_Tij}) between two arbitrarily chosen nodes (see figure caption) as a function of $p$. 
We first determine the set of r-nodes, i.e.\ those nodes $k$ for which $R_k\neq 0$ with $N$ independent Bernoulli trials (one for each node), with success probabilities $0.1, 0.2,\ldots,0.8, 1$. This generates populations of respectively $10\%,20\%,\ldots,80\%,100\%$ of r-nodes. In Fig.\ \ref{BA_network}b the resetting to the r-nodes is performed with
uniform probability, that is $R_k =R$ if $k$ is an r-node and $R_k =0$ otherwise. Each curve corresponds to a certain fraction of r-nodes. Observe that if just $10\%$ of the nodes are r-nodes, the MFPT is increasing with $p$, i.e.\ resetting does not increase efficiency of a searcher. The same behavior is observed if we choose only one 
or $10$ r-nodes. In these cases of low r-node population, it is possible that the walker at each reset is relocated to a node far from the target, which may extend the time span until the target is found.
The behavior is suddenly changing when we increase the r-node population to $20\%$
and higher. In these cases the MFPT is decreasing with $p$ where the decrease is monotonic up to about $60\%$ of r-node population. For higher r-node populations the MFPT first decreases and then increases again with $p$, so that an optimal resetting rate $p_*$ for which the MFPT is minimal exists. The larger the r-node population is, the smaller $p_*$ becomes. This effect can be interpreted that for large r-node populations the walker is repeatedly relocated to nodes far away from the target node.

If the resetting probabilities are proportional to the degree $R_k \propto K_k$, and all other parameters remain the same (see Fig.\ \ref{BA_network}c), we have that for sufficiently large r-node populations ($>10\%$ ) the MFPT decreases monotonically with $p$, and the decrease becomes less pronounced the larger the proportion of the r-nodes is.
In addition, one can observe that the absolute values of the MFPTs for the same $p$ in the $m=2$ WS network, which we discuss in more details subsequently, are much greater than in the BA network, reflecting the large world feature of the WS network (see Fig.\ \ref{WS_network}).

In order to get a more global picture on the efficiency of a searcher to find a target in the BA graph, we plot in Fig. \ref{Kemeny_BA_WS} the Kemeny constant $\mathcal{K}(p)$ (Eq.\ (\ref{Kemeny_result})) as a function of $p$. Note that the Kemeny constant is monotonically decreasing with $p$. In other words, the expectation of the MFPT for a randomly chosen target $j$ (with NESS probability $P_j(p;R_R,\ldots,R_N;\infty)$, see (\ref{Kemeny_def})) 
decreases monotonically, approaching the minimum value of about $500=N$ for $p \to 1-$, which corresponds to the optimal one in the CC network (see (\ref{kemen_CC})).

The values of the MFPT in Fig. \ref{BA_network} are significantly smaller than the optimal value of the Kemeny constant $N=500$. Our interpretation is that the chosen target node $200$ in this network has a relatively large degree, $K_{200}=13$, compared to the average degree $\langle K\rangle \approx 7.94$. We also notice that the values of the MFPT approach constant values for $p\to 1$ only if the t-node is in the set ${\cal L}_R$ of r-nodes. We explain this feature hereafter in details.
It has to be pointed out 
that in the limit $p=1$ the walker visits only the r-nodes with ``steps'' (the resets) drawn from $\mathbf{R}$.
This holds independently of the network topology: the walker moves in a completely connected structure which is constituted by the set of r-nodes.
If the t-node is not an r-node, then the MFPT should tend to infinity as $p\to 1$, 
as the t-node cannot be reached. On the contrary if the t-node is an r-node, the MFPT approaches a constant value (see (\ref{important_case_set_tarket})). This behavior can clearly be identified in Figs. \ref{BA_network} and \ref{WS_network} where in both networks our choice is such that the t-node $j=200$ is not an r-node
for  one, $10$ r-nodes and for $10\%$ of r-node population. 
For all other cases  ($20\% - 100\%$ r-node fractions), the t-node $200$ is within the set ${\cal L}_R$ of r-nodes.
We illustrate these distinct behaviors in frame \ref{BA_network}(d) representing \ref{BA_network}(b) in a larger window.
Consider more closely the frame \ref{BA_network}(b) 
for the cases in which the t-node $j\in {\cal L}_R$ with uniform resetting probabilities $R_j=1/n_R$ where $n_R$
denotes the number of r-nodes. The finite value of the MFPT is such that $\lim_{p \to 1}\langle T_{100,200} \rangle = n_R$ (see  (\ref{important_case_set_tarket})) independently of the network topology (see also the WS case in Fig. \ref{WS_network})(b)).
For the mentioned cases of $20\%,...,100\%$ of r-nodes, we have the MFPT values $n_R = 100,\ldots,500$ 
at $p=1$, see again Fig. \ref{BA_network}(b). 
This behavior can be seen in frame \ref{BA_network}(e) which shows
$\langle T_{100,200}\rangle$ as a function of the r-node fractions for some fixed values of $p$. The cyan points correspond to the MFPTs at $p \approx 1$ and the horizontal lines represent the corresponding values of $n_R$. 
We use for the latter the same color codes as in frames \ref{BA_network}(b,d). 
We explain the non-monotonous behavior of the MFPT which occurs 
for $10\%$ to $20\%$ of r-nodes (see panel \ref{BA_network}(e)) by the fact that in the first case 
the t-node is not r-node, whereas is in the latter. In the $10\%$ case the target can only be found between consecutive resets where the walker performs “local steps” to next neighbor nodes drawn from $\mathbf{W}$. For increasing $p$ it becomes more and more “difficult” to reach it and for $p = 1$ impossible. This explains the divergence of the MFPT
for $p\to 1$ in the cases in which the t-node $200 \notin{\cal L}_R$ (see panel \ref{BA_network}(d)).
On the contrary, if $j \in {\cal L}_R$, the probability to reach the t-node increases with $p$ as long as $n_R$ is sufficiently small. 
To give 
evidence for this behavior, we depict in Fig. \ref{new_fig2} the MFPT to the t-node $200$ (the departure node is $100$).
 We use identical parameters and BA realization as in Fig. \ref{BA_network}. The dashed curves indicate the MFPTs, where the t-node $200$ is excluded from the set of r-nodes.
 The continuous curves correspond to situations where the t-node is included into the set of r-nodes.
For $n_R<N$ the ergodicity of the walk is lost at $p=1$. The behaviors observed in the WS graph considered next, are consistent with the above interpretation (see Fig. \ref{WS_network}).
\begin{figure}[t!]
%
\centerline{
\includegraphics[width=0.5\textwidth]{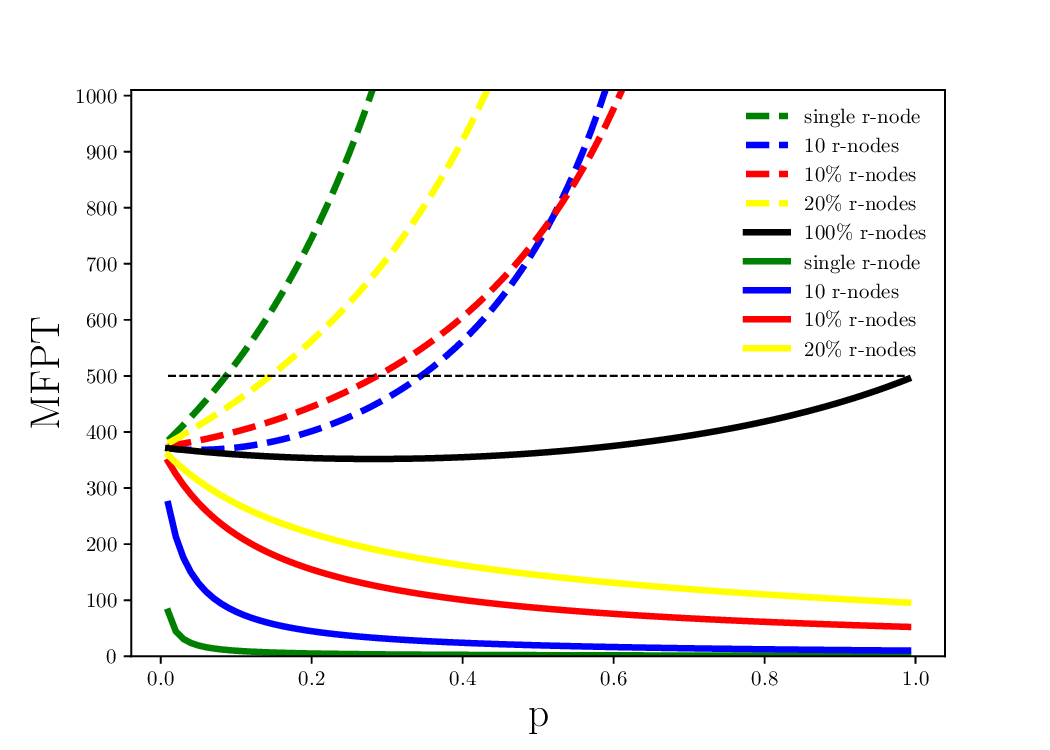}
\includegraphics[width=0.5\textwidth]{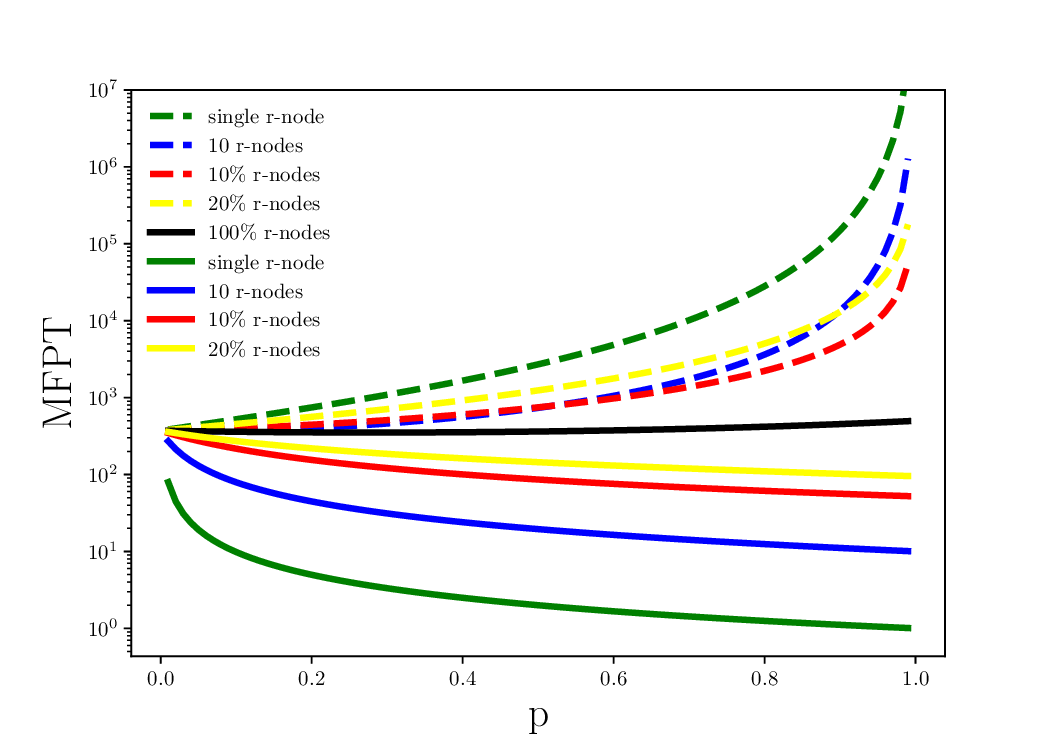}
}
\vspace{-2mm}
\caption{ MFPT $\langle T_{100,200}\rangle$ for uniform resetting probabilities $R_r=1/n_R$.
The situations of $j=200 \notin {\cal L}_R$
and $j=200 \in  {\cal L}_R$ are indicated by dashed and continuous curves, respectively. All parameters and the network are identical as in Fig. \ref{BA_network}.  }
\label{new_fig2}
\end{figure}

\paragraph{Watts-Strogatz graph} 
%
%
%
\begin{figure}[t!]
\centerline{
\includegraphics[width=0.8\textwidth]{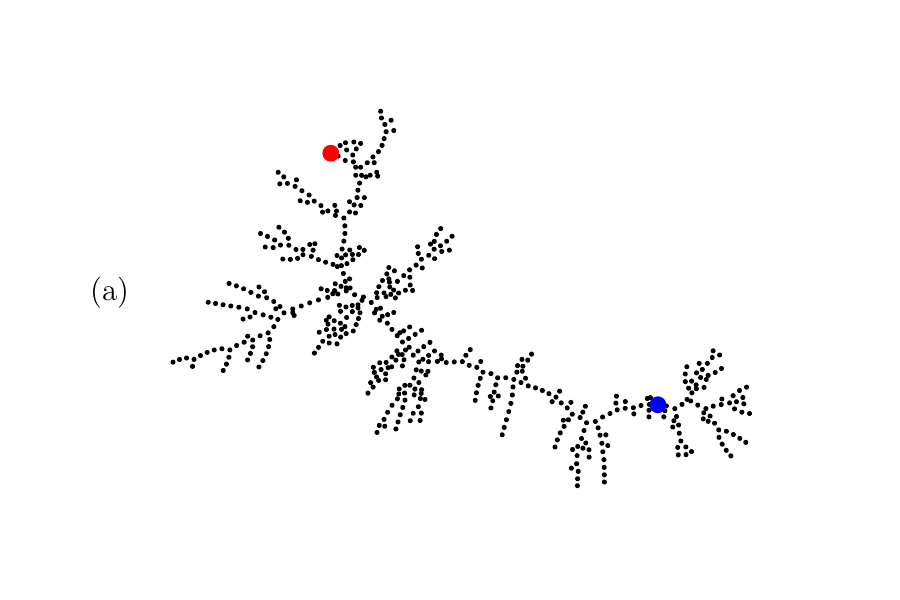}
}
\vspace{-2cm}
\centerline{
\includegraphics[width=0.5\textwidth]{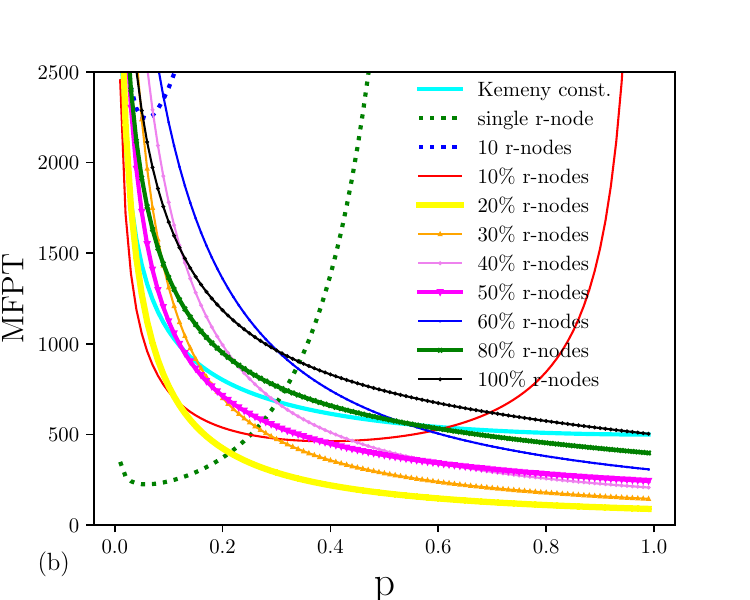}
\includegraphics[width=0.5\textwidth]{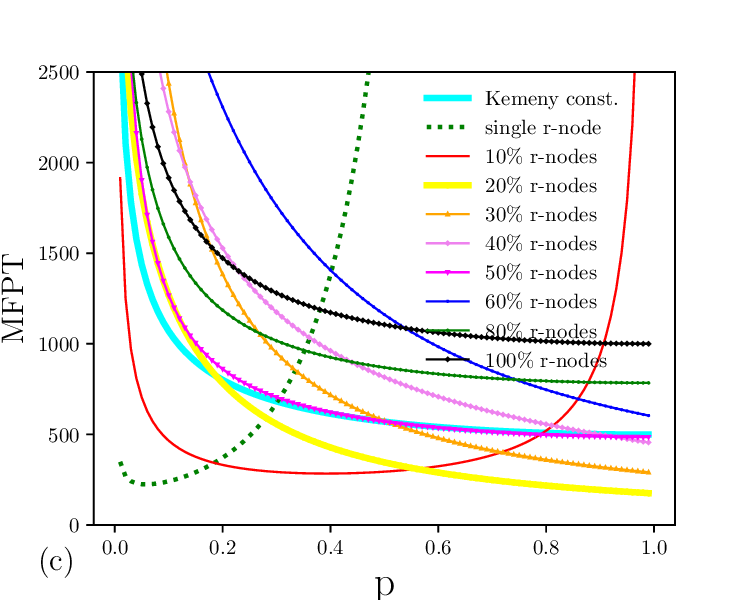}
}
\vspace{-2mm}
\caption{(a) Random search with resets between the starting node (blue) and the target node (red) in a Watts-Strogatz graph with $N=500$ nodes, attachment parameter $m=2$ (large-world) and rewiring probability $0.7$.
The Kemeny constant (Eq.\ (\ref{Kemeny_result})) is pictured in  (b) and (c) in cyan color.
We plot $\langle T_{ij} \rangle$ (Eq.\ (\ref{represntation_Tij})) with respect to $p$ for the starting node $i=100$ and the target node $j=200$. 
(b) Same resetting scenarios and color codes as in Fig.\ \ref {BA_network}(b) with uniform resetting probabilities ($R_r=R$) to the r-nodes. (c) Same resetting scenarios and color codes as in Fig.\ \ref{BA_network}(c) with preferential resetting to the r-nodes ($R_r \propto K_r$).
In both frames, the t-node $200$ is excluded from ${\cal L}_R$ for
one, $10$ and $10\%$ of r-nodes, whereas it is included for all other cases. }
\label{WS_network}
\end{figure}
\begin{figure}[t!]
\centerline{
\includegraphics[width=0.9\textwidth]{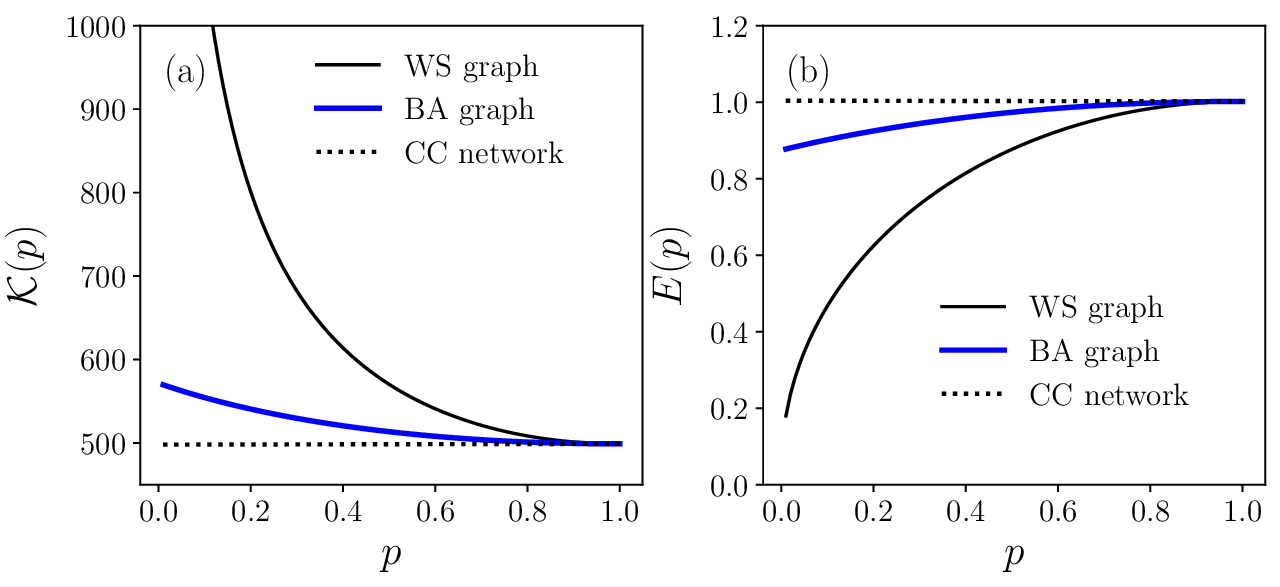}
}
\vspace{-2mm}
\caption{Comparison of the (a) Kemeny constant ${\cal K}(p)$ and (b) search efficiency $E(p)$ of Eq.\ (\ref{searcher_efficieny}) for the dynamics with resets in the Barabási-Albert and 
Watts-Strogatz graphs. The parameters are those of Figs.\ \ref{BA_network} and \ref{WS_network}. The dotted lines denote
the CC graph with ${\cal K}_{CC} \approx N=500$ and $E_{CC}(p) \approx 1$ ($N \gg 1$).}
\label{Kemeny_BA_WS}
\end{figure}
To generate a WS graph \cite{WattsStrogatz1998} one starts with a ring of $N$ nodes where each node is symmetrically connected with $m <N$ neighbor nodes. In this way,
each node has $2m$ connections. In a second step, each connection of a node is replaced, with constant probability $p_0$, 
by a connection to randomly chosen nodes (avoiding link duplication and self connections). If the rewiring probability $p_0=0$, then 
the ring with constant degree $2m$ is retained, whereas if $p_0=1$ the resulting graph is an Erd\"os-Renyi graph.
The WS graph is generally small world, except for $m=2$, which we use here (see Fig.\ \ref{WS_network}), for which the graph exhibits 
large distances between nodes.

Let us explore now the effect of the resetting when the searcher operates in a large world structure such as the WS network for $m=2$ (Fig.\ \ref{WS_network}). We observe that for both resetting strategies (see frames \ref{WS_network}(b,c), uniform and preferential relocations, respectively), the MFPT between the considered nodes exhibits strong decay with increasing resetting rate, as soon as there are sufficiently many r-nodes. This holds for both strategies. One observes that the effect of the resetting on the search efficiency is much more pronounced in large world networks than in small world networks (compare Figs.\ \ref{BA_network}(b,c) and \ref{WS_network}(b,c)). The presence of resetting significantly reduces the MFPT in some cases. We explain this feature subsequently.
The main observation is that the presence of resetting is especially useful to increase the search efficiency in large world architectures, as it reduces the MFPT between distant nodes. This effect persists for $p\to 1$
if the t-node is contained in the set of r-nodes.
In addition, one intuitively infers that in this way resets occur to nodes which are not that far from the 
target node, allowing the searcher to find it more quickly. In Figs.\ \ref{WS_network}(b,c), we have also represented the dependence of the Kemeny constant on $p$ (cyan curves). One can see that
for $p\to 1$ the minimum value $N=500$ is approached, with maximal search efficiency $E(p) \to 1$.

In order to give a visual impression of this complex dynamics, we present videos of two realizations of this walk, which can be launched by clicking online on the words "simulation video" marked hereafter.  
In these animations, the nodes visited by the walker are colored in orange. The resets occur with uniform probability to any node of the network (corresponding to $100\%$ of r-node population, cf.\ Fig.\ \ref{WS_network}(b)) and the runtime in both simulations is $300$. The 
\href{https://drive.google.com/file/d/1bdEzrpH1NDYq8IOCSRkAtuqQLS1_YnKQ/view}{\textcolor{violet}{\it first simulation video}} is characterized by the resetting rate $p=0.05$. In the \href{https://drive.google.com/file/d/1B1AdQfj-J7-JdK-BuubxgP3jKV68BtNo/view?usp=sharing}{\textcolor{violet}{\it second simulation video}}, $p=0.8$. We determined numerically the number of visited nodes during the runtime. In the first simulation the number of visited nodes is $105$, and in the second one $212$. This empirically shows the faster exploration of the graph (and hence shorter MFPTs) when the resetting rate is increased.

In Fig.\ \ref{Kemeny_BA_WS}, we show the Kemeny constants and search efficiencies of the above considered graphs.
We remark once again that the
resetting most efficiently reduces the MFPT for the WS graph. One can also note that the values of the Kemeny constant are similar to those of the respective MFPTs (Figs.\ \ref{BA_network}, \ref{WS_network}).
For both the WS and the BA graph the optimal resetting rate $p_*$ minimizing the Kemeny constant is close to one, whereas does not exist for the CC graph.
\section{First hitting statistics for non-Markovian RRPs}
\subsection{Basic model}
\label{FPH-quantities}
Here, our aim is to investigate the first hitting statistics of a Markovian random walker for arbitrary RRPs with focus on non-Markovian ones and arbitrary set of target nodes. 
The walker is killed (i.e.\ removed from the network) as soon as it reaches one of the t-nodes. 
We refer to this time instant as the first hitting time (FHT) \cite{Redner2001}. 
We exclude killing at $t=0$.
Clearly, the FHT is the number of steps the walker needs to hit one of the t-nodes for the first time.
In a searching framework, the hitting of a t-node represents a \emph{success}.
We refer to \cite{PalSandev2023} for a continuous space model with killing targets and to  
\cite{Masuda-et-al2017,BasolasNicosia2021} for models based on absorbing states, and consult also \cite{ChenYe2022}.
Recall that $\mathbf{W}$ ($W_{ij}=A_{ij}/K_i$) is the transition matrix of the considered Markov walk without killing.
We talk about \emph{Auxiliary Walk} (AW) 
if the killing condition on t-nodes is included.
In this view, it is useful to consider
the matrix ${\widetilde {\bf W}}$ related to the AW killed at the FHT.
Let ${\cal B}$ denote a set of t-nodes and introduce the diagonal matrix $\mathbf{\Theta}({\cal B})$
containing ``$1$'' in the diagonal positions which correspond to t-nodes and ``$0$'' elsewhere, thus
$\Theta(i;{\cal B}) = \sum_{b \in {\cal B}}\delta_{ib}$ defines an indicator function such that
$\Theta(i;{\cal B})=1$ if $i \in {\cal B}$ and $\Theta(i;{\cal B})=0$ otherwise.
We further introduce $\mathbf{1}-\mathbf{\Theta}({\cal B})$ as the diagonal matrix with the complementary entries.
For the AW we have
\beq
\label{transition_mat_mod} 
\ds {\widetilde {\bf W}} =   \mathbf{W} \cdot [\mathbf{1} - \mathbf{\Theta}({\cal B})]  , 
\hspace{1cm} {\widetilde W}_{ij} = W_{ij}[1-\Theta(j;{\cal B})]
\eeq
with entries ${\widetilde W}_{ij} =0$ if $j \in {\cal B}$ and
${\widetilde W}_{ij} = W_{ij} =  A_{ij}/K_i$ otherwise.
Observe that the matrix ${\widetilde {\bf W}}$ is defective, i.e.\ not all rows are properly normalized 
$\sum_{j=1}^N {\widetilde W}_{ij} =1-\sum_{j\in {\cal B}} W_{ij} = q_i \leq 1 $ where $q_i<1$ for some $i$ (labeling neighbors of t-nodes). 
The defective feature of ${\widetilde {\bf W}}$ is inherited by its powers ${\widetilde {\bf W}}^k$, $k \geq 1$, which we 
consider in more details later on. Consult Appendix \ref{fist_encouter_paths} for technical details.
In this construction, the walker can start the walk from a t-node without being killed at $t=0$. 
This behavior is ensured by defining
\beq
\label{case_kzero}
{\widetilde {\bf W}}^0 = \mathbf{1} = [ \delta_{ij} ] , \hspace{1cm} i,j =1,\ldots, N. 
\eeq
An AW walker that departs from a t-node is killed only at its first return (unless hitting another t-node before).
The entry $[{\widetilde {\bf W}}]^t_{ij}$ denotes the probability that the AW walker, starting from node $i$, reaches node $j$ alive at time $t$ (i.e.\ it has never hit a t-node before time $t$). For $j \in {\cal B}$ per construction, this probability is null for all $t\geq 1$. The sum $\sum_{j=1}^N[{\widetilde {\bf W}}]^t_{ij}$ is therefore the probability that the walker is somewhere on the network and alive at time $t$. We will interpret this quantity as the survival probability of the AW walker (see (\ref{survival_monkey})). Further, this coincides with the probability that the Markovian walker 
with propagator $\mathbf{W}^t$ has never hit any t-node up to time $t$. 

To include the feature that the walker is killed in a reset to t-node as well, we 
modify the resetting matrix as
\beq
\label{mod_reset_mat}
{\widetilde {\bf R}} =  \mathbf{R} \cdot  
[\mathbf{1} - \mathbf{\Theta}({\cal B})]  , \hspace{1cm}{\widetilde R}_{ij} = R_j [1-\Theta(j;{\cal B})]
\eeq
where ${\widetilde R}_{ij} =R_j$ if $j \notin {\cal B}$ and ${\widetilde R}_{ij} =0$ for $j \in {\cal B}$.
We observe the following properties:
$\sum_{j=1}^N{\widetilde W}_{ij} = q_i \leq 1$ and $\sum_{i=1}^N q_i < N$, i.e.\ ${\widetilde {\bf W}}$ is defective for any target ${\cal B} \neq \emptyset$. On the other hand we have
$\sum_{j=1}^N{\widetilde R}_{ij} =r_i = 1 -\sum_{j \in {\cal B}} R_j \leq 1$, where $\sum_{i=1}^N r_i \leq N$, and the equality is attained if ${\cal L}_R \cap {\cal B} = \emptyset$, which  represents the case in which ${\widetilde {\bf R}} = \mathbf{R}$ remains row-stochastic.

Now we introduce the survival probability of the AW walker on the network (probability that it has not hit the target up to and including time $t$) with starting node $i$ as\footnote{For brevity, we will sometimes use the notation $M_i = \sum_{j=1}^N M_{ij}$.}
\beq
\label{survival_monkey}
\Lambda_i(t) = \sum_{j=1}^N [{\widetilde {\bf W}}^t]_{ij} =  [{\widetilde {\bf W}}^t]_i  = 1 - \sum_{r=1}^t \chi_i(r) , \hspace{1cm} t \in \mathbb{N}_0
\eeq
which is non-increasing with time. Further we set $\Lambda_i(0)=1$ for all $i=1,\ldots, N$ (including $i \in {\cal B}$, 
see (\ref{case_kzero})) and note $0 \leq \Lambda_i(t) \leq 1$.
The AW propagator ${\widetilde {\bf W}}^t$ contains the statistics on the first hitting to the target ${\cal B}$ performed by the Markovian walker $\mathbf{W}^t$
(see Appendix \ref{fist_encouter_paths}
for more details).
In (\ref{survival_monkey}) we introduced the first hitting time PDF 
$\chi_i(t) = \mathbb{P}[T_i=t]$
as the probability that the walker hits one of the t-nodes for the first time at time $t$.
The first hitting time PDF for a starting node $i$ then reads
\beq
\label{MFPD}
\chi_i(t)= \Lambda_i(t-1) - \Lambda_i(t) , \hspace{1cm} t \geq 1, 
\eeq
where we considered that $\chi_i(0) = 0$ as $T_i \geq 1$ ($\Lambda_i(0)=1$). Its GF is given by
\beq
\label{MFPD_GF}
{\bar \chi}_i(u) = 1-(1-u){\bar \Lambda}_i(u) , \hspace{1cm} |u| \leq 1
\eeq
with ${\bar \Lambda}_i(u)= [(\mathbf{1}-u{\widetilde {\bf W}})^{-1}]_i$.
In Appendix \ref{fist_encouter_paths} we show indeed that $\chi_i(t) = F_{ib}(t)$ coincides with (\ref{first-passage_proba}) for a single t-node $b$.
We point out that the AW propagator
${\widetilde {\bf W}}^t$ does not have the canonical representation (\ref{purely_markovian}). In particular,
it has no unique eigenvalue equal to one, thus reflecting the lack of a proper row-normalization. The FHT PDF $\chi_i(t)$ is a properly normalized PDF only if $\Lambda_i(\infty)=0$, which indeed holds true for the AW propagator ${\widetilde {\bf W}}^t$ 
as ${\widetilde {\bf W}}^t \to \mathbf{0}$ for $t\to \infty$. We have

\begin{thm}\label{theorem}
    The matrix ${\widetilde {\bf W}}$ has spectral radius $\rho({\widetilde {\bf W}}) = \max(|{\tilde \lambda}_k|) < 1$.
\end{thm}

\begin{proof}
We first note that
the Markovian walker $\mathbf{W}^t$ realizes an irreducible and aperiodic Markov chain \cite{Newman2010,fractionalbook-2019,NohRieger2004,VanMiegem2011,Birkhoff1931,Lefebvre2007}\footnote{Recall, a non-negative  $N\times N$ matrix $\mathbf{M}$ is referred to as {\it primitive} (synonymously ``irreducible aperiodic'' or ``ergodic''), if there is a positive integer $k_0'$ such that the matrix powers $\mathbf{M}^k$ for $k\geq k_0'$ are positive matrices $[\mathbf{M}^k]_{ij} >0$, ($i,j=1,\ldots,N$) \cite{Meyer2000}.}. 
Further, although $\mathbf{W}$ is primitive, ${\widetilde {\bf W}}$ is not, as matrix powers of the latter preserve the zero columns which label t-nodes.
Let $k_0 \in \mathbb{N}$ be such that $\mathbf{W}^{k_0}$ is positive.
We introduce the auxiliary matrix $\mathbf{Q} = [Q_{ij}] = \mathbf{W}^{k_0}\cdot [\mathbf{1} - \mathbf{\Theta}({\cal B})]$. 
It holds
\beq
\label{inqual_ergo}
\sum_{j=1}^N [{\widetilde {\bf W}}^{k_0}]_{ij} \leq Q_i= \sum_{j=1}^N Q_{ij} =1- \sum_{j\in {\cal B}} [\mathbf{W}^{k_0}]_{ij} < 1 , \hspace{1cm} i = 1, \ldots, N,
\eeq
with $Q_i < 1$ for all rows $i= 1, \ldots, N$.
Let $Q=\max(Q_i)$ and consider the matrix powers ${\widetilde {\bf W}}^{k_0 n}$, $n\in \mathbb{N}$,
for $n \to \infty$.
This leads, for all $i=1,\ldots, N$, to the inequalities 
\beq
\label{power_tilde_W}
 \Lambda_i(nk_0) =  \sum_{j=1}^N [{\widetilde {\bf W}}^{k_0 n}]_{ij}  \leq  \sum_{j=1}^N[\mathbf{Q}^n]_{ij} =
\sum_{k_1=1}^N \ldots \sum_{k_{n-1}=1}^N  Q_{i k_1} \ldots Q_{k_{n-2},k_{n-1}}  \sum_{j=1}^N Q_{k_{n-1},j} 
   \leq  Q^n  \to 0.
\eeq
This concludes the proof.
\end{proof}

Note that, for our goals, it is not necessary to consider further spectral details of ${\widetilde {\bf W}}$.

From the geometric decay of (\ref{power_tilde_W}) it follows that all the moments of the first hitting time are finite. Indeed, they are given by 
\beq
\label{FHT_moments}
\big\langle  T_i^m \big\rangle = (-1)^m \frac{d^m}{ds^m}{\bar \chi}_i(e^{-s})  \bigg|_{s=0} , \hspace{1cm} m \in \mathbb{N}.
\eeq
For $m=1$ one gets the mean first hitting time (MFHT). Indeed, from the relation $ {\bar \Lambda}_i(u) =(1-{\bar \chi}_i(u))/(1-u)$, we have
\beq
\label{MFPT_non_Markov}
\big\langle T_i \big\rangle = \frac{d}{du}{\bar \chi}_i(u)  \bigg|_{u=1}  =  {\bar \Lambda}_i(u) \bigg|_{u=1} = 
[(\mathbf{1}-{\widetilde {\bf W}})^{-1}]_i
\eeq
which has to coincide, for a single target, with (\ref{MFPT})
(in the case of a very small resetting rate).
Observe that $\mathbf{1}-{\widetilde {\bf W}}$ is invertible (Theorem \ref{theorem}). 
We numerically verified the accordance of (\ref{MFPT}) (for a very small resetting rate $p$) and (\ref{MFPT_non_Markov}) for a single 
t-node $j=200$ and starting node $i=100$.
For the WS network of Fig.\ \ref{WS_network}
we obtain, with $p=10^{-9}$, $\langle T_{100,200}\rangle \approx 32398$ from both Eqs.\ (\ref{MFPT}) and (\ref{MFPT_non_Markov}), and the accordance is maintained for different values of the parameters. Further, for the considered realization of the BA network, we obtain $\langle T_{100,200}\rangle \approx 387.69$.
We point out that \eqref{power_tilde_W}
tells us that in an ergodic Markovian walk any target is hit with probability one. We will show  later that this remains true under resetting as long as the resulting associated walk is ergodic. We point out that in the Markov case, i.e. when the RRP is a Bernoulli process, the killing t-nodes approach has a connection to absorbing Markov chains. In particular consult \cite{GrinsteadSnell2006}, Theorem 11.3. We outline more details in Appendix \ref{AW_AMC}. 
\\ 
The situation may change 
to $\Lambda_i(\infty) >0$
under resetting when additional conditions are met, destroying ergodicity of the associated walk. 
We will consider at the end of this section such a case and explore the issue of ergodicity more closely.
With these considerations, we are ready to establish the defective propagator of the {\it auxiliary walk under resetting -- (AWR)} as
\beq
\label{walk_with_resets_def}
\begin{array}{lll}
\ds  {\widetilde {\bf P}}_{\text{AWR}}({\widetilde {\bf W}}, {\widetilde {\bf R}} ;t)  & =  {\widetilde {\bf W}}^t  \big\langle \Theta(\Delta t_1-1-t) \big\rangle   &  \\ \\ &  \ds
+  \sum_{n=1}^{\infty}  \big\langle \Theta(J_n,t,J_{n+1})  ({\widetilde {\bf W}}^{\Delta t_1-1}\cdot {\widetilde {\bf R}}) \cdot 
 \ldots \cdot ({\widetilde {\bf W}}^{\Delta t_n-1}\cdot {\widetilde {\bf R}}) \cdot {\widetilde {\bf W}}^{t-J_n} \big\rangle 
 &  
 \end{array}
\eeq
which is supported on $t \in \mathbb{N}_0$ and fulfills the initial condition $ {\widetilde {\bf P}}_{\text{AWR}}({\widetilde {\bf W}}, {\widetilde {\bf R}} ;0) = \mathbf{1}$. At each integer time instant, the AWR-walker either performs a step or it is reset until being killed by visiting a t-node. 
We will refer to (\ref{walk_with_resets_def}) as the {\it survival propagator (SP)}
as it contains the first hitting statistics of given t-node set ${\cal B}$ 
of the (row-stochastic) propagator ${\bf P}(\mathbf{W}, \mathbf{R} ;t)$ (\ref{walk_with_resets}).
The latter is recovered when we replace in (\ref{walk_with_resets_def})
${\widetilde {\bf W}}$ with $\mathbf{W}$ and ${\widetilde {\bf R}}$ with $\mathbf{R}$. 
We have $[\mathbf{W}^k \cdot \mathbf{R}]_{ij} =\sum_{k=1}^N [\mathbf{W}^k]_{kj} R_j  = R_j $ which leads us to
$${\bf W}^{\Delta t_1-1}\cdot {\widetilde {\bf R}} \cdot \ldots \cdot {\bf W}^{\Delta t_n-1}\cdot {\bf R}= \mathbf{R}^n =\mathbf{R}, \hspace{1cm} n \geq 1. $$ 
We refer from now on (\ref{walk_with_resets}) to as \emph{associated walk}. 
The SP only considers the sample paths which have not reached a t-node up to and including time $t$. These sample paths occur in the AWR with the same probabilities as in the associated walk.
The SP entry $[{\widetilde {\bf P}}_{\text{AWR}}({\widetilde {\bf W}}, {\widetilde {\bf R}} ;t)]_{ij}$ is the probability that the AWR walker, starting from node $i$, reaches node $j$ alive at time $t$.
In (\ref{walk_with_resets_def}), the memory terms $ ({\widetilde {\bf W}}^{\Delta t_1-1}\cdot {\widetilde {\bf R}}) \cdot \ldots \cdot ({\widetilde {\bf W}}^{\Delta t_n-1}\cdot {\widetilde {\bf R}})$ contain the history of killing events which happened prior to the $n$-th reset. 

Worthy of mention is the situation when ${\cal L}_R \subseteq {\cal B}$ where the walker is always reset to a t-node. 
In this case one has ${\widetilde {\bf R}}= \mathbf{0}$, thus (\ref{walk_with_resets_def}) contains only the first term.
\begin{Proposition}
The GF, w.r.t.\ $t$, of (\ref{walk_with_resets_def}) is
\beq
\label{Gf_transition matrix_AW}
{\bar {\bf P}}_{\text{AWR}}({\widetilde {\bf W}}, {\widetilde {\bf R}} ;u ) = \left({\mathbf 1} - u {\bar g}(u{\widetilde {\bf W}}) \cdot  {\widetilde {\bf R}}\right)^{-1}
\cdot \frac{{\mathbf 1} - {\bar \psi}(u {\widetilde {\bf W}})}{{\mathbf 1} - u{\widetilde {\bf W}} }   , \hspace{1cm} |u| \leq 1 .
\eeq
\end{Proposition}

\begin{proof}
Note that the matrix ${\widetilde {\bf W}}$ is not invertible as $\text{rank}({\widetilde {\bf W}})= \text{rank}(\mathbf{W}) - N_B < N$ where $N_B$ is the number of t-nodes, and that
(\ref{Gf_transition matrix_AW}) 
contains only non-negative powers of ${\widetilde {\bf W}}$. To make this more obvious we introduced the additional GF
\beq
\label{supplement_GF_v}
{\bar g}(v)= \left\langle v^{\Delta t -1} \right\rangle = \sum_{t=1}^{\infty} \psi(t) v^{t-1}  , \hspace{1cm} |v|\leq 1
\eeq
with $\psi(t)=g(t-1)$. The GF, w.r.t.\ $t$, of $\psi(t)v^{t-1}$ is $u{\bar g}(uv)$ and it is present in (\ref{Gf_transition matrix_AW}) as
the matrix GF 
$$ \big\langle {\widetilde {\bf W}}^{\Delta t -1} u^{\Delta t} \big\rangle 
=  \sum_{k=1}^{\infty}u^k \psi(k) \, {\widetilde {\bf W}}^{k-1}   = u\, {\bar g}(u{\widetilde {\bf W}}) . $$
We notice that replacing ${\widetilde {\bf W}}$ and ${\widetilde {\bf R}}$ with $\mathbf{W}$ and $\mathbf{R}$, respectively, 
in (\ref{Gf_transition matrix_AW}), we recover the corresponding GF 
(\ref{Gf_transition matrix}) of the associated walk (see Appendix \ref{AWR_derivations}).
Then, recalling that $J_{n+1}=J_n +\Delta t_{n+1}$,  and using that $\Delta t_j$ are IID copies of $\Delta t$ we arrive at
\begin{align*}
& \ds  \sum_{t=0}^{\infty} u^t \big\langle \Theta(J_n,t,J_{n+1}) ({\widetilde {\bf W}}^{\Delta t_1-1}\cdot {\widetilde {\bf R}})\cdot \ldots \cdot ({\widetilde {\bf W}}^{\Delta t_n-1}\cdot {\widetilde {\bf R}}) \cdot {\widetilde {\bf W}}^{t-J_n}   \big\rangle \\
& = \ds   \big\langle u^{\Delta t_1} {\widetilde {\bf W}}^{\Delta t_1-1} \big\rangle \cdot  {\widetilde {\bf R}}  \cdot \ldots \cdot 
 \big\langle u^{\Delta t_n} {\widetilde {\bf W}}^{\Delta t_n-1}  \big\rangle  \cdot  {\widetilde {\bf R}} \cdot \big\langle \sum_{k=0}^{\Delta t_{n+1}-1}  u^k {\widetilde {\bf W}}^k \big\rangle  \\
& = \ds \left\{  u{\bar g}(u{\widetilde {\bf W}}) \cdot {\widetilde {\bf R}} \right\}^n  \frac{{\mathbf 1} - {\bar \psi}(u {\widetilde {\bf W}})}{{\mathbf 1} - u{\widetilde {\bf W}} }.  
\end{align*}
For $n=0$ we have
$$
\sum_{t=0}^{\infty} u^t {\widetilde {\bf W}}^t  \big\langle \Theta(\Delta t_1-1-t) \big\rangle = \frac{{\mathbf 1} - \langle (u {\widetilde {\bf W}})^{\Delta t_1}\rangle }{{\mathbf 1} - u{\widetilde {\bf W}}}
 = \frac{{\mathbf 1} - {\bar \psi}(u {\widetilde {\bf W}})}{{\mathbf 1} - u{\widetilde {\bf W}} } .
$$
Summing these expressions over $n$ takes us to (\ref{Gf_transition matrix_AW}).
\end{proof}
In addition, (\ref{Gf_transition matrix_AW}) allows us to establish a renewal equation for the SP (see Appendix \ref{AWR_derivations}):
\beq
\label{renewal_AWR}
{\widetilde {\bf P}}_{\text{AWR}}({\widetilde {\bf W}}, {\widetilde {\bf R}} ;t) = {\widetilde {\bf W}}^t \Phi^{(0)}(t)+ 
\sum_{k=1}^t \psi(k) {\widetilde {\bf W}}^{k-1} \cdot 
{\widetilde {\bf R}} \cdot {\widetilde {\bf P}}_{\text{AWR}}({\widetilde {\bf W}}, {\widetilde {\bf R}} ;t-k) .
\eeq
This relation reflects the non-Markovianity of the AWR (unless the RRP is Bernoulli, 
see Appendix \ref{AWR_derivations}, (\ref{Bernoulli_AWR}), (\ref{Bernoulli_reset_AWR_propagator})), 
containing
${\widetilde {\bf P}}_{\text{AWR}}({\widetilde {\bf W}}, {\widetilde {\bf R}} ;t-k)$ which is in fact related to the history of the killing process.
Again, if we replace the couple ${\widetilde {\bf W}}, {\widetilde {\bf R}}$ with $\mathbf{W}, \mathbf{R}$, we retrieve
the renewal equation (\ref{renewal_structure}) of the associated walk.
Inverting (\ref{Gf_transition matrix_AW}) we arrive at the SP with
survival probability (with starting node $i$) 
\beq
\label{AWR_survival}
\Lambda_i(t) =  [{\widetilde {\bf P}}_{\text{AWR}}({\widetilde {\bf W}}, {\widetilde {\bf R}} ;t)]_i  =  
 \sum_{j=1}^N[{\widetilde {\bf P}}_{\text{AWR}}({\widetilde {\bf W}}, {\widetilde {\bf R}} ;t)]_{ij} 
\eeq
and the FHT PDF is then given by (\ref{MFPD}).
Later in the paper, we will focus on the MFHT for a starting node $i$, which is given by
\beq
\label{AW_wit_resets}
\big\langle T_i \big\rangle= 
[{\bar {\bf P}}_{\text{AWR}}({\widetilde {\bf W}}, {\widetilde {\bf R}} ;u)]_i\bigg|_{u=1} .
\eeq
\subsection{Some remarks on ergodicity}
\label{ergo}
An important question is whether a random searcher eventually hits a target. A feature which is closely related to this issue is the ergodicity of the associated walk.
\begin{rmk}\label{remremII}
Recall the notion of ergodicity \cite{Birkhoff1931,Lefebvre2007}. Roughly speaking, an ergodic associated walker looses memory of the starting node and reaches any target ${\cal B}$  with probability one. In a finite network an ergodic walk is also recurrent.
We call an associated walk (\ref{walk_with_resets}) ergodic, if it has a unique stationary propagator (be it a NESS or not) which is a positive matrix, and is independent of the starting nodes, i.e.
\beq
\label{ergodicity}
\mathbf{P}_{ij}(\infty)  = [\mathbf{R} \cdot {\bar f}(\infty,\mathbf{W})]_{ij} = P_j(\infty) > 0, \qquad
i,j=1,\ldots, N
\eeq
where ${\bar f}(t,\mathbf{W})$ is the GF (\ref{Gf_B}) for the matrix argument $\mathbf{W}$ (see (\ref{NESS})).
An associated walk for which the NESS $\mathbf{P}(\infty)$ contains zero entries is non-ergodic. We consider such a class later on.
We show hereafter that if the infinite time limit of the SP (\ref{walk_with_resets_def}), for any target ${\cal B}$, is the zero matrix,
then the corresponding associated walk is ergodic 
and vice versa. Equivalently, 
for any choice of ${\cal B}$ one has the spectral radius
\beq
\label{power_tilde_W-criterion}
\rho({\bar g}({\widetilde {\bf W}})\cdot {\widetilde {\bf R}}) \leq  \max([{\bar g}({\widetilde {\bf W}})\cdot {\widetilde {\bf R}}]_i) < 1.
\eeq
Recall that ${\bar \psi}(u) = u {\bar g}(u)$. See (\ref{supplement_GF_v}) where ${\bar \psi}(u)$ is the GF (\ref{Gf_spi}) of the PDF $\psi(t)$ (\ref{discret-PFD}).
The property (\ref{power_tilde_W-criterion}) includes, but is not limited to, the following cases:
\begin{description}
    \item [(a)] The matrix ${\bar g}(\mathbf{W})$ is positive;
    \item [(b)] The resetting matrix is positive $R_j>0$ for every $j=1,\ldots,N$, i.e. ${\cal L}_R$ comprises the whole network;
    \item [(c)] The matrix $\mathbf{R}\cdot {\bar g}(\mathbf{W})$ is positive.
\end{description}
\end{rmk}
In case {\bf (a)} the associated walker can explore the whole network between the resets for any starting node and can hit any target with strictly positive probability prior to the first reset. In this case ${\bar g}({\widetilde {\bf W}})$ and ${\bar g}({\widetilde {\bf W}})\cdot {\widetilde {\bf R}}$ have spectral radii smaller than one (Theorem \ref{theorem}). 

If condition {\bf (b)} is met, ${\widetilde {\bf R}}$ and ${\bar g}({\widetilde {\bf W}})\cdot {\widetilde {\bf R}}$ have spectral radii smaller than one (again, Theorem \ref{theorem}). The associated walker is able to explore the network by resets and can hit any target with strictly positive probability at the first reset.

In case {\bf (c)} the associated walker can hit any node of the network with strictly positive probability after the first reset.
If {\bf (a)} or {\bf (b)} or both are fulfilled, then {\bf (c)} holds true, but {\bf (c)} does not imply {\bf (a)} or {\bf (b)}.
If {\bf (c)} holds true, both ${\widetilde {\bf R}} \cdot {\bar g}({\widetilde {\bf W}})$ and therefore ${\bar g}({\widetilde {\bf W}})\cdot {\widetilde {\bf R}}$ have spectral radii smaller than one (again, Theorem \ref{theorem}).

To see that {\bf (c)} is sufficient for ergodicity, consider
\beq
\label{ergodicity-proof}
\begin{array}{clr}
\ds \mathbf{P}(\mathbf{W},\mathbf{R};\infty) = \ds  \mathbf{R} \cdot {\bar f}(\infty,\mathbf{W}) & = \ds 
\frac{1}{\langle \Delta t \rangle} \mathbf{R} \cdot \sum_{r=1}^{\infty} \psi(r) \sum_{k=0}^{r-1} \mathbf{W}^k & \\ \\ & = \ds 
\frac{1}{\langle \Delta t \rangle} \mathbf{R}\cdot {\bar g}(\mathbf{W}) + \frac{1}{\langle \Delta t \rangle}\mathbf{R} \cdot \sum_{r=2}^{\infty} \psi(r) \sum_{k=0}^{r-2} \mathbf{W}^k &
\end{array}
\eeq
which fulfills the requirements of (\ref{ergodicity}).
In the case in which $\psi$ and $g$ are fat-tailed, $\langle \Delta t \rangle$ diverges, and {\bf (a)} and {\bf (c)} are satisfied. Then, the associated walk is ergodic and $P_j(\infty)$ coincides with the equilibrium distribution (see (\ref{switch-back}), (\ref{NESS}) together with (\ref{definilimit})). We investigate this class in subsequent Section \ref{Sibuya_ARW}.
From (\ref{power_tilde_W-criterion}), the matrix ${\mathbf 1} - {\bar g}({\widetilde {\bf W}})\cdot {\widetilde {\bf R}}$ in (\ref{Gf_transition matrix_AW}) is invertible. This leads to the following properties as a hallmark of ergodicity of the corresponding associated walk: for any choice of ${\cal B}$
$${\widetilde {\bf P}}_{\text{AWR}}({\widetilde {\bf W}}, {\widetilde {\bf R}}; \infty) = (1-u) {\bar {\bf P}}_{\text{AWR}}({\widetilde {\bf W}}, {\widetilde {\bf R}}; u)\bigg|_{u=1}
 = 0.$$  In addition, the MFHTs
 $$\langle T_i \rangle = [ {\bar {\bf P}}_{\text{AWR}}({\widetilde {\bf W}}, {\widetilde {\bf R}}, u)]_i\bigg|_{u=1} < \infty$$ and all higher moments exist
 for all starting nodes and any target ${\cal B}$. This can be easily seen in (\ref{Gf_transition matrix_AW}) by the existence of derivatives of any order $n$ with respect to $s$ ($u=e^{-s}$), when (\ref{power_tilde_W-criterion}) is fulfilled. From this observation, we infer that
 the SP approaches zero at least geometrically, and so do the survival probabilities.
We point out that (\ref{power_tilde_W-criterion}) represents a sufficient criterion for ergodicity of the associated walk if it holds for any choice of ${\cal B}$.

\subsection{A class of non-ergodic associated walks}
\label{5825}
We consider here the class of associated walks for which the conditions {\bf (a)}, {\bf (b)} and {\bf (c)} (see Section \ref{ergo}) are not fulfilled. In this case the walker cannot explore the entire network.
We focus on the case in which the set of r-nodes ${\cal L}_R$ and that of the t-nodes ${\cal B}$ are disjoint, ${\cal L}_R \cap {\cal B} = \emptyset$, and have a sufficiently ``large'' network distance. Then, there exists a $T \in \mathbb{N}$ 
such that $[\mathbf{W}^k]_{sj} = 0$ for $k \leq T$ for all
$s \in  {\cal L}_R$ and all $j \in {\cal B}$. In other words, the shortest path connecting r- and t-nodes is longer than $T$ steps. Further, we assume that the time interval between consecutive resets $\Delta t_T$ is such that $ \max (\Delta t_T) = T $.
It has a discrete PDF $\psi_T(t)$ 
supported on $\{1,\ldots, T \}$ with ${\bar \psi}_T(u)= u{\bar g}_T(u) = \sum_{t=1}^T \psi_T(t) u^t$. Observe that
${\bar \psi}_T(1)= \sum_{t=1}^T \psi_T(t) = 1$. We refer this RRP to as T-RRP. 
The mean waiting time between consecutive resets is $\langle \Delta t_T \rangle =\sum_{r=1}^T r\psi_T(r) \leq T$ (the equality is attained in the deterministic $T$-periodic case -- see the end of this section). 
In this construction, the Markovian walker never performs more than $T-1$ successive steps drawn from 
$\mathbf{W}$ between consecutive resets, which limits the capacity to explore the network.
The associated walk has a NESS of the form (see (\ref{NESS}))
\beq
\label{NESS_T}
\mathbf{P}(\mathbf{W},\mathbf{R};\infty) =  \frac{1}{\langle \Delta t_T\rangle } \mathbf{R} \cdot \sum_{k=1}^T\psi_T(k) \sum_{r=0}^{k-1} \mathbf{W}^r
\eeq
with $\mathbf{P}(\mathbf{W},\mathbf{R};t) = \mathbf{P}(\mathbf{W},\mathbf{R};\infty)$ for $t\geq T$.
This matrix has identical rows and some zero entries. 
The occupation probabilities of all t-nodes are null for $t\geq T$, as (\ref{NESS_T}) contains only matrix powers of orders $r<T$ with $[\mathbf{W}^r]_{sj}=0$, thus $P_{sj}(\mathbf{W},\mathbf{R};\infty) =0$ 
for $j\in {\cal B}$, $s \in {\cal L}_R$. 
The smaller $T$ is chosen, the more localized the walker remains close to r-nodes after the first reset.
In turn, for starting nodes which are sufficiently ``close'' to a t-node, one expects that the walker may hit the target prior to the first reset with strictly positive probability. 
Let us check this assertion. Notably, since ${\cal L}_R \cap {\cal B} =\emptyset$, we have
${\widetilde {\bf R}} = \mathbf{R}$ maintaining row-stochasticity, and we will make use of
\beq
\label{make_use_of}
 {\widetilde {\bf R}}\cdot  {\widetilde {\bf W}}^k = \mathbf{R}\cdot \mathbf{W}^k  , \hspace{1cm} k \leq T .
\eeq
The rows labeling r-nodes of ${\widetilde {\bf W}}^k$ and $\mathbf{W}^k$ ($k \leq T$) coincide.
Therefore, $\mathbf{R} \cdot {\bar g}_T({\widetilde {\bf W}}) = \mathbf{R} \cdot {\bar g}_T(\mathbf{W}) $ is row-stochastic.
In particular, one has
\beq
\label{non-ergodic}
\begin{array}{clr}
{\widetilde {\bf R}} \cdot {\bar g}_T({\widetilde {\bf W}}) \cdot {\widetilde {\bf R}} 
= & \mathbf{R} \cdot {\bar g}_T(\mathbf{W}) \cdot \mathbf{R}  = \mathbf{R} & \\ \\
(\mathbf{R} \cdot  {\bar g}_T({\widetilde {\bf W}}))^n = &  \mathbf{R} \cdot  {\bar g}_T(\mathbf{W}) & \\ 
  & & \hspace{1cm} n \in \mathbb{N}, \\ 
( {\bar g}_T({\widetilde {\bf W}}) \cdot \mathbf{R} )^n = &  {\bar g}_T({\widetilde {\bf W}}) \cdot \mathbf{R}  & 
\end{array}
\eeq
remaining constant for $n\to \infty$. Since these relations are independent of $n$, we infer that 
the matrices $\mathbf{R}$,  $\mathbf{R} \cdot  {\bar g}_T({\widetilde {\bf W}})$ and ${\bar g}_T({\widetilde {\bf W}}) \cdot \mathbf{R}$ have only eigenvalues $0$ and $1$. In particular, $\text{rank}[\mathbf{R} \cdot  {\bar g}_T(\mathbf{W}]= \text{rank}[\mathbf{R}] = 1$, i.e.\ the latter have $N-1$ eigenvalues $0$ and a unique eigenvalue $1$.
As a hallmark of non-ergodicity, one has for our choice of the target, the spectral radii
\beq
\label{spectral_radius}
\rho({\bar g}_T({\widetilde {\bf W}}) \cdot \mathbf{R}) = \rho( \mathbf{R} \cdot {\bar g}_T({\widetilde {\bf W}}) )  = 1 .
\eeq
For our aims, it is not necessary to explore further spectral details. 
More generally speaking, an associated walk under T-RRP resetting is non-ergodic 
if it exists a choice of ${\cal B}$  for which (\ref{spectral_radius}) holds true.
We consider now the first hitting statistics and evaluate (\ref{Gf_transition matrix_AW}) as
\beq
\label{GF_AWR_short_T}
\begin{array}{clr}
\ds {\bar {\bf P}}_{\text{AWR}}({\widetilde {\bf W}}, \mathbf{R}; u) & = \ds \left( \mathbf{1} + \frac{ u{\bar g}_T(u {\widetilde {\bf W}})} {1-{\bar \psi}_T(u)}  \cdot \mathbf{R}\right) \cdot 
\frac{\mathbf{1}-{\bar \psi}_T(u{\widetilde {\bf W}})}{\mathbf{1} -u{\widetilde {\bf W}}} & \\ \\
 & =  \ds \left( \mathbf{1} + \frac{ u{\bar g}_T(u {\widetilde {\bf W}})} {1-{\bar \psi}_T(u)}  \cdot \mathbf{R}\right) \cdot \sum_{k=1}^T\psi_T(k) \sum_{r=0}^{k-1} u^r {\widetilde {\bf W}}^r   &
 \end{array}  \hspace{1cm} |u| < 1 .
\eeq
In the second line we used the T-RRP features (\ref{make_use_of}), (\ref{non-ergodic}) taking us to 
$$ [u{\bar g}_T(u {\widetilde {\bf W}}) \cdot \mathbf{R}]^n = [{\bar \psi}_T(u)]^{n-1}  u {\bar g}_T(u {\widetilde {\bf W}}) \cdot \mathbf {R} .$$
The large time limit of the survival probabilities yields
\beq
\label{survival_non-ergodic}
\Lambda_i(\infty) = (1-u) \sum_{j=1}^N[{\bar {\bf P}}_{\text{AWR}}({\widetilde {\bf W}}, \mathbf{R}; u)]_{ij}\bigg|_{u\to 1-}= 
\sum_{j=1}^N [{\bar g}_T({\widetilde {\bf W}})]_{ij} = \left\{ \begin{array}{cl}   1 , 
& i \in {\cal L}_R \\ \\
 c \leq 1 , & i \notin {\cal L}_R \end{array} \right. 
\eeq
where $\Lambda_i(t)=\Lambda_i(\infty)$ remain constant and strictly positive for $t \geq T$ and for all starting nodes, as no more hits to t-nodes may happen after and from the first reset on.
An AWR searcher starting from an r-node almost surely stays alive forever.
In turn, a walker starting from a node $i \notin {\cal L}_R$, located less than $T$ steps away from ${\cal B}$,
hits the target prior to the first reset with positive probability $1-\Lambda_i(\infty) = \sum_{t=1}^{\infty} \chi_i(t) < 1$. Note that for these starting nodes $\chi_i(t)$ is defective.
On the other hand, we have (see (\ref{GF_AWR_short_T}))
\beq
\label{T-RRP-MFHT}
\langle T_i \rangle = \sum_{j=1}^N[{\bar {\bf P}}_{\text{AWR}}({\widetilde {\bf W}}, \mathbf{R}; u)]_{ij}\bigg|_{u=1} = 
\infty
\eeq
for all starting nodes $i$, even for those less than $T$ steps away from the target. Our interpretation is, for all starting nodes,
the existence of escape paths avoiding the target prior to the first reset and therefore never hitting it.
Expression
(\ref{survival_non-ergodic}) can also be obtained directly from (\ref{renewal_AWR}) by setting  there $\Phi^{(0)}(t)=0$ holding for $t\geq T$
and using that
$\mathbf{R}\cdot {\widetilde {\bf P}}_{\text{AWR}}(\mathbf{R},{\widetilde {\bf W}},t-k) = \mathbf{R}\cdot \mathbf{P}(\mathbf{W},\mathbf{R},t-k)$ remains forever row-stochastic.
One has $[{\widetilde {\bf W}}^{k-1} \cdot\mathbf{R}\cdot \mathbf{P}_{\text{AWR}}(\mathbf{R},{\widetilde {\bf W}},t)]_i = [{\widetilde {\bf W}}^{k-1}]_i$ for all $i=1,\ldots, N$ and $k\leq T$ reflecting the feature that the complete history of the killing events takes place prior to the first reset and is ``recorded'' in $[{\bar g}_T({\widetilde {\bf W}})]_i$. 
To visualize such a non-ergodic walk, we present an animation which can be viewed by clicking online on the
\href{https://drive.google.com/file/d/1QR5iqNYazWPA52pL_D9lRhcOm66jD2yp/view}{\textcolor{violet}{\it simulation video}}.
In this simulation we have chosen $T=50$ and a uniform distribution $\psi_T(t) = \frac{1}{T}$, $t \in \{1,\ldots,T\}$ for the inter-reset times. We consider this walk on a realization of a WS graph with $300$ nodes and the other parameters identical to the simulations of Section \ref{first_passage}. The r-node (with label $100$) coincides with the departure node and the t-node (with label $200$) is more than $50$ steps away from the departure node and hence can never be reached. One can see that the walker is confined in a region of $50$ steps away from the r-node. The deterministic case of $T$-periodic resets with $T=50$ (Section \ref{det_lim}) is visualized by clicking online on this \href{https://drive.google.com/file/d/1Pa6DZkW7G5MGGBe6O_eWTsJLD2E3bZI5/view}{\textcolor{violet}{\it simulation video}} for the same network and identical parameters. The runtime is $300$ in both cases and less than $40$ nodes are visited. Note that the smaller $T$ and 
$\langle \Delta t_T\rangle $ are, the more confined the walker is to a region close to the starting point.
\subsection{Deterministic limit}
\label{det_lim}
The T-RRP class also contains the $T$-periodic deterministic limit, which is simply retrieved for $\psi_T(t)=\delta_{t,T}$. The  $T$-periodic associated walk propagator then reads
\beq
\label{periodic}
\mathbf{P}(\mathbf{W}, \mathbf{R}; t) = \Theta(T-1-t) \mathbf{W}^t + \Theta(t-T) \mathbf{R} \cdot \mathbf{W}^{\, t \hspace{0.1cm} {\rm mod} \hspace{0.1cm} T}  , \hspace{1cm} t \in \mathbf{N}_0
\eeq
and hence does not take a steady state for $t\to \infty$. 
The large time limit (\ref{NESS_T}), in this case, is the time-average over one period
$T^{-1}\mathbf{R}\cdot \sum_{k=0}^{T-1} \mathbf{W}^k$.
A special periodic case is again $T=1$ where in each time instant a reset occurs
(that is ${\cal N}(t)=t$ almost surely) with ${\bar \psi}(u)=u$ (${\bar g}(u)=1$), which we already considered previously.
In that case
(\ref{Gf_transition matrix_AW}) boils down to
\beq
\label{Non_ergodic}
{\bar {\bf P}}_{\text{AWR}}({\widetilde {\bf W}},\mathbf{R} ; u ) = \frac{\mathbf{1}}{\mathbf{1}-u\mathbf{R}}  = 
\frac{1}{1-u} \mathbf{R},
\eeq
with $\Lambda_i(t)=1$ for all starting nodes with
$\mathbf{P}(\mathbf{W}, \mathbf{R}; t) = \mathbf{R}$ ($t\geq 1$) where the associated walker forever navigates in ${\cal L}_R$ and never hits the target.
\section{Sibuya RRP}
\label{Sibuya_ARW}
A rather interesting case in the class of non-Markovian RRPs is that in which the time interval between consecutive resets is Sibuya distributed \cite{Sibuya1979}. The Sibuya distribution is fat-tailed with infinite mean. In this case no NESS exists, and the associated walk is ergodic
(see also Appendix \ref{standard_Sib}).
Recall that the probability that no reset occurs up to and including time $t$ is 
\cite{GSD-Squirrel-walk2023}
\beq
\label{persistence_prob}
\Phi_{\alpha}^{(0}(t) = \frac{\Gamma(t+1-\alpha)}{\Gamma(1-\alpha)\Gamma(t+1)} , \hspace{0.5cm} t \in \mathbb{N}_0, \hspace{0.5cm} \alpha \in (0,1), \hspace{1cm} \Phi_{\alpha}^{(0}(t) \sim \frac{t^{-\alpha}}{\Gamma(1-\alpha)} ,\hspace{0.5cm}  t \to \infty,
\eeq
with $\Phi_{\alpha}^{(0}(0)=1$. The symbol $\sim$ denotes asymptotic equality. The Sibuya resetting rate is explicitly obtained as (see (\ref{resettying_sib}))
\beq
\label{resettying_sib_rate}
 {\cal R}_{\alpha}(t) = -\delta_{t0} + \frac{\Gamma(\alpha+t)}{\Gamma(\alpha)\Gamma(t+1)}  , \hspace{1.5cm} {\cal R}_{\alpha}(t)\sim \frac{t^{\alpha-1}}{\Gamma(\alpha)}  \to 0 ,\hspace{1cm}  t \to \infty,
\eeq
with ${\cal R}_{\alpha}(0)=0$ and $0\leq {\cal R}_{\alpha}(t) \leq 1$. 
We plot in Fig.\ \ref{Sib_resetting_rate} ${\cal R}_{\alpha}(t)$ as a function of $\alpha$ for some range of $t$. For fixed $t$ the resetting rate ${\cal R}_{\alpha}(t)$ is monotonically increasing approaching $1$ from the left for $\alpha \to 1-$,  corresponding to the limit ${\cal N}_{1-}(t) \to t$ of the Sibuya RRP counting process ${\cal N}_{\alpha}(t)$. On the other hand, for $\alpha \to 0+$ we have 
${\cal R}_{\alpha}(t) \to 0+$ (as $\Gamma(\alpha) \to \infty$).
\begin{figure}[t!]
\centerline{
\includegraphics[width=0.75\textwidth]{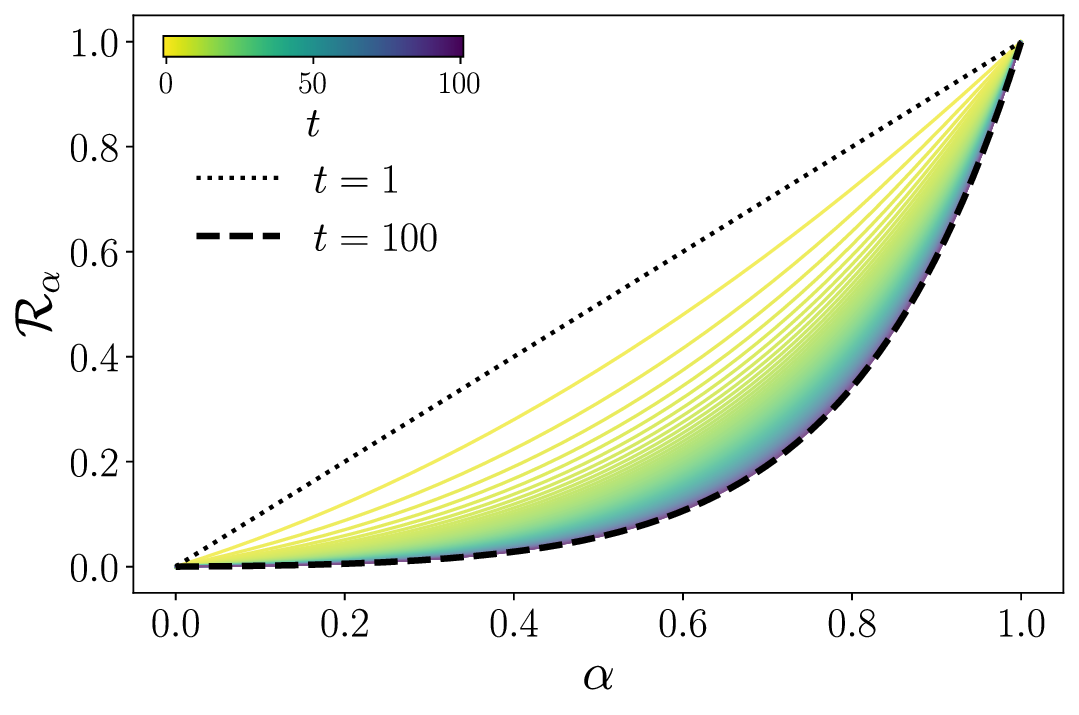}
}
\vspace{-3mm}
\caption{Sibuya resetting rate ${\cal R}_{\alpha}$ of Eq.\ (\ref{resettying_sib_rate}) as a function of $\alpha$ and different values of $t$. The numerical results for $1<t<100$ are coded with the colorbar, whereas the two extremal cases $t=1$ and $t=100$ discontinuous lines are depicted.}
\label{Sib_resetting_rate}
\end{figure}
\noindent
The propagator of the associated Sibuya walk is then explicitly given by (\ref{propagator_time}) together with  (\ref{persistence_prob}) and (\ref{GF_V}). Instructive here is only its large-time asymptotics which, together with (\ref{switch-back}) and (\ref{definilimit}), is obtained as
\beq
\label{large_time_Sibuya}
\begin{array}{clr}
\ds \mathbf{P}^{(\alpha)}(t) & \ds \sim \mathbf{R} \cdot {\bar f}_{\alpha}(t,\mathbf{W}) & \\  \\
P_{ij}^{(\alpha)}(t) & \ds \sim \frac{K_j}{\sum_{r=1}^NK_r} + \frac{t^{\alpha-1}}{\Gamma(\alpha)} \sum_{m=2}^N (1-\lambda_m)^{\alpha-1}\sum_{k=1}^N R_k \langle k| \phi_m\rangle\langle{\bar \phi}_m|j\rangle  &
\end{array} \hspace{0.5cm}  t \to \infty, 
\eeq
and does not depend on the starting nodes. The approach toward the equilibrium is similar to that of the
confined CTRW on the line, see \cite{MetzlerKlafter2000}.
The only exception occurs for
$\mathbf{R} = |\phi_1\rangle\langle{\bar \phi}_1|$ where the decay to the stationary distribution is a geometrically tempered power-law (see 
(\ref{switch-back})).
As mentioned, there is no NESS here and the stationary equilibrium distribution is approached slowly by a power-law as a landmark of the non-Markovianity of the Sibuya counting process. 
In the following, we restrict ourselves to investigate the first hitting statistics. To that end, we evaluate the GF of Sibuya SP (\ref{Gf_transition matrix_AW}),
\beq
\label{Gf_transition matrix_AW_Sibuya}
{\bar {\bf P}}_{\text{AWR}}^{(\alpha)}({\widetilde {\bf W}}, {\widetilde {\bf R}} ;u ) = \left(\mathbf{1} - {\widetilde {\bf W}}_{\epsilon}^{-1} \cdot \{
\mathbf{1} - [\mathbf{1}- u {\widetilde {\bf W}}]^{\alpha} \} \cdot {\widetilde {\bf R}} \right)^{-1} 
\cdot (\mathbf{1}- u {\widetilde {\bf W}})^{\alpha -1}  , \hspace{1cm} |u| \leq 1,
\eeq
where, from a numerical point of view, it is convenient to use the pseudo-inverse of ${\widetilde {\bf W}}$, which is defined by
\beq
\label{pseudo_inverse}
{\widetilde {\bf W}}_{\epsilon}^{-1} =
\Re \, \frac{\mathbf{1}}{{\widetilde {\bf W}} + {\rm i}\,\epsilon \mathbf{1} } = 
\frac{{\widetilde {\bf W}}}{{\widetilde {\bf W}}^2 + \epsilon^2\,\mathbf{1} } , \hspace{1cm} \epsilon \to 0.
\eeq
where  one has that
$ {\widetilde {\bf W}}_{\epsilon}^{-1} {\bar \psi}(u{\widetilde {\bf W}}) \to u{\bar g}(u{\widetilde {\bf W}}))$
for $\epsilon \to 0+$.

There are two noteworthy Markovian limits. For $\alpha \to 0+$, where  no resets occur almost surely, one recovers the Markovian AW with ${\bar {\bf P}}_{\text{AWR}}^{(0+)}({\widetilde {\bf W}}, {\widetilde {\bf R}} ;u) =[1-u{\widetilde {\bf W}}]^{-1}$ and ${\widetilde {\bf P}}_{\text{AWR}}^{(0+)}({\widetilde {\bf W}}, {\widetilde {\bf R}} ;t)= {\widetilde {\bf W}}^t$.
The second limit $\alpha \to 1-$ is also Markovian and corresponds to the trivial counting renewal process discussed previously where at each time instant a reset occurs. Formula (\ref{Gf_transition matrix_AW_Sibuya}) then boils down to
${\bar {\bf P}}_{\text{AWR}}^{(1-)}(u) = 
(\mathbf{1}-  u {\widetilde {\bf R}})^{-1}$ with ${\bf P}_{\text{AWR}}^{(1-)}(t) = {\widetilde {\bf R}}^t$. This limit is non-ergodic unless 
${\cal L}_R$ comprises the entire graph.
From (\ref{AW_wit_resets}), we evaluate the Sibuya MFHT 
$\big\langle T_i(\alpha) \big\rangle  = [{\bar {\bf P}}_{\text{AWR}}^{(\alpha)}({\widetilde {\bf W}}, {\widetilde {\bf R}} ;1 )]_i$  and take the average over all starting nodes $i$ to define the global Sibuya MFHT 
\beq
\label{global_MFHT}
\big\langle T_{\alpha} \big\rangle = \frac{1}{N}\sum_{i=1}^N \big\langle T_i(\alpha) \big\rangle .
\eeq
We show in Figs.\ \ref{Sib_MFHT}(a,c) the global MFHT (\ref{global_MFHT}) as a function of $\alpha$ for the 
above considered WS and BA graphs and some proportions of uniformly distributed t-nodes, determined (as previously) by independent Bernoulli trials. 
In the same settings, plots (b,d) in Fig. \ref{Sib_MFHT} show the MFHTs with starting node $i=100$.
\begin{figure}[t!]
\centerline{
\includegraphics[width=1.0\textwidth]{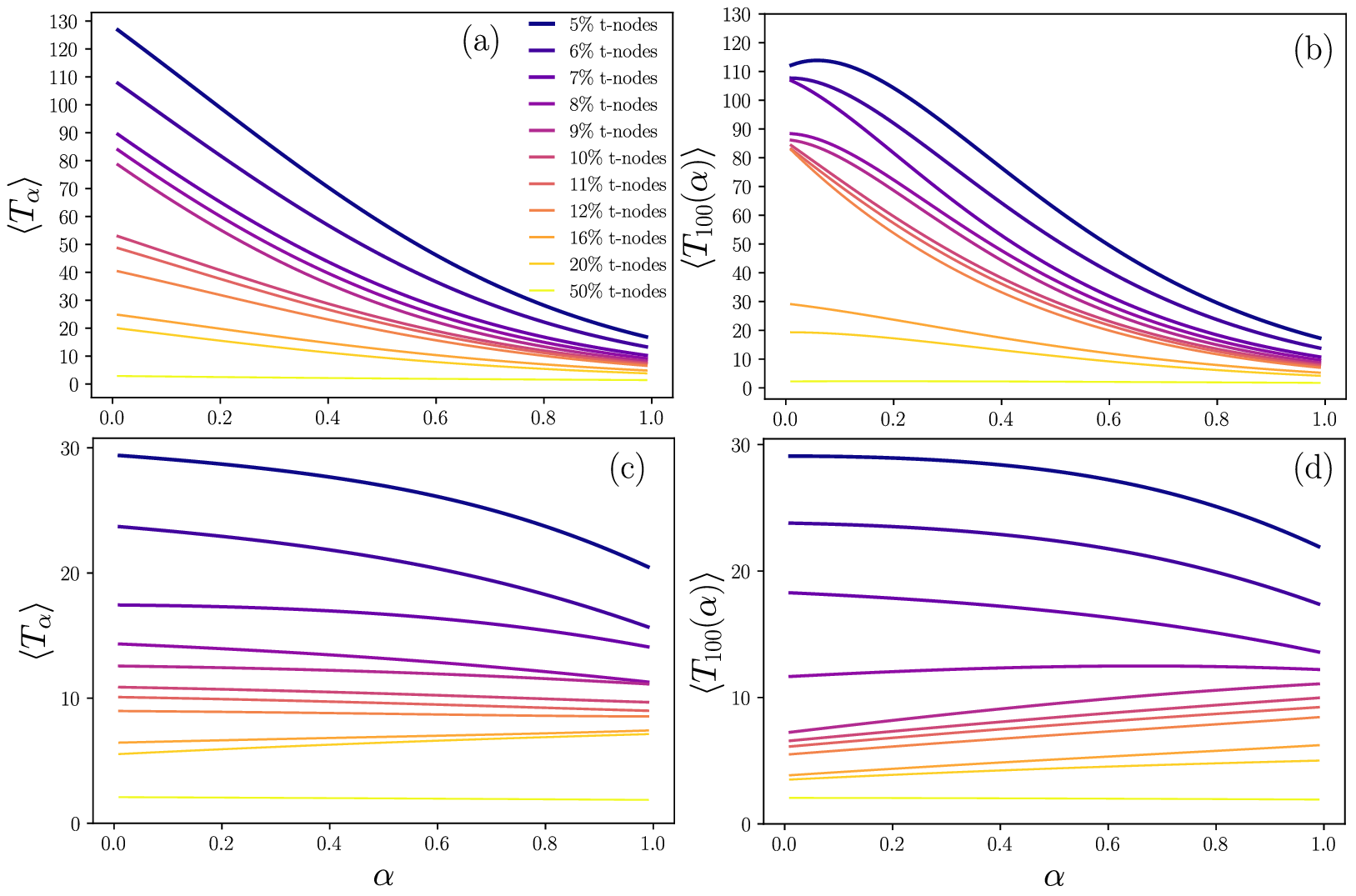}
}
\caption{Global MFHT and MFHT as functions of the Sibuya index $\alpha$ with uniform
resetting probabilities to any node for some proportions of uniformly distributed t-nodes in the Watts-Strogatz and Barabási–Albert networks with size $N=500$. Frame {\bf (a)}: $\langle T_{\alpha} \rangle$ in Eq.\ (\ref{global_MFHT}) as a function of $\alpha$ for the Watts-Strogatz network. Frame {\bf (b)}: $\langle T_{100}(\alpha)\rangle$ as a function of $\alpha$ for the same network as in (a) with  starting node $i=100$. Frames {\bf (c)} and {\bf (d)} present the same analysis for the Barabási–Albert network. The parameters of the considered graphs are identical as in Figs.\ \ref{BA_network} and \ref{WS_network}.
In all frames, we use the colors for the curves in panel {\bf (a)} to code the proportion of t-nodes.}
\label{Sib_MFHT}
\end{figure}
In all frames of this figure, the resets are performed with uniform probability $R_j=1/N$ to any node of the network. 
In the (large world) WS graph, $\langle T_{\alpha} \rangle$ exhibits a pronounced decay as $\alpha$ increases, in particular for small proportions of t-nodes. A minimum is reached for $\alpha \to 1-$. We interpret this decrease as follows: frequent resetting to any node brings the walker, on average, quickly close to a t-node.  
For fixed $\alpha$, the global 
MFHT decays monotonically as the proportion of t-nodes increases. One observes that for high t-node proportions ($\approx 50\%$)  the global MFHT and the MFHT for a specific starting node ($i=100$)
take the minimal constant value of about one, independent of $\alpha$ in both graph types. This makes sense since in such a configuration each node is likely to be either a t-node or a neighbor of a t-node.
In Fig. \ref{Sib_MFHT}(b) one can see that the MFHT in the WS graph
for starting node $i=100$ has a similar behavior as the global MFHT of panel (a), in particular for large $\alpha$. In this case, the memory of the starting node is weakened by frequent resets (as soons as the Sibuya resetting rate increases monotonically with $\alpha$, see Fig.\ \ref{Sib_resetting_rate}). For small t-node proportions ($\approx 5\%$) we observe that the MFHT first slightly increases (as effect of the particular location of the starting node) and then falls off with very similar values as the global MFHT. A comparison of the frames (a) and (b) shows that the global MFHT captures well the essential features in the large world WS graph.

In the small world BA network considered in the frame \ref{Sib_MFHT}(c), $\langle T_{\alpha} \rangle$ decays slightly with respect to $\alpha$,
where the decrease becomes less pronounced for larger proportions of t-nodes ($>10\%$) (turning into a slight increase for proportions lager than $16\%$), maintaining approximately the values of the limit of the Markovian walk without resets $\alpha \to 0+$. 
A comparison of frames (c) with (d) shows that the behavior of the MFHT $\langle T_{100}(\alpha) \rangle$ is rather similar to the global MFHT. For larger t-node proportions ($>10\%$) $\langle T_{100}(\alpha) \rangle$ slightly increases with respect to $\alpha$, maintaining values close to the Markovian limit $\alpha \to 0+$ without resetting as well.
We explain this very weak dependence on $\alpha$ by the small world feature of the BA network (short average distances between nodes). 
Our interpretation is that, relocating the walker in a small world structure with a large t-node proportion, in the average, does not change the
distance to a given target significantly. 
On the other hand, for sparse t-node distributions, frequent resetting to any node of the network is still advantageous in reducing the first hitting times.

In the considered WS graph, the absolute values of the global MFHT for small $\alpha$ (long waiting times between resets) and in the Markovian limit $\alpha=0$ are much larger compared to those in the BA graph. This clearly reflects the fact that the WS graph has larger average distances between nodes as the BA graph (see again the representations of these graphs in Figs.\ \ref{BA_network}(a) and \ref{WS_network}(a), respectively). Therefore, it makes sense that a random searcher 
needs in the average more time to hit the same target in the WS graph than in the BA graph. Frequent resets to any node clearly compensate the disadvantage of longer distances. 
These results together with those obtained for Bernoulli resets give a consistent picture,
telling us that the presence of resetting reduces mean first hitting times most significantly in large world structures with sparse targets. To enhance this effect, the relocation nodes should be widespread over the network.
\section{Conclusions and further directions}
\label{Conclusions}
We developed a discrete time model for a Markovian random walker under 
renewal resetting moving in a finite undirected, connected, ergodic, network. 
Our approach considers Markovian and non-Markovian RRPs.
From the concept of backward recurrence time, we derived an exact formula for the propagator of the associated walk (see (\ref{propagator_time})). 
We applied these results to Markovian resetting and discussed some related first passage quantities.
We observe a non trivial behavior of the MFPT both as a function of the resetting rate or the proportion of r-nodes.
We showed that the presence of resetting reduces the MFPT most efficiently in the considered large world WS graph, whereas the impact of resetting is less effective in small world networks such as the considered BA
graph. As expected, the effect of resetting disappears in completely connected structures (see Fig.\ \ref{Kemeny_BA_WS}).

We focused on  different resetting scenarios.
We developed a framework based on killing the walker at the time of reaching target nodes to study the first hitting statistics in the case of non-Markovian RRPs. We derived an exact formula (see (\ref{Gf_transition matrix_AW})) for the GF of the SP (\ref{walk_with_resets_def}) from which information on first hitting statistics for the associated walk can be obtained (see (\ref{walk_with_resets}), (\ref{propagator_time})).
These exact results are useful to analyze search strategies for a wide range of random walks under non-Markovian RRPs in complex environments, and various relocation scenarios. In particular, the SP allowed us to formulate some sufficient conditions for which the associated walk is ergodic (see Remark \ref{remremII}). 
In order to show that the associated walks are not always ergodic, we introduced a class of RRPs for which 
the walker after the first reset is trapped in a certain region, preventing the complete exploration of the network.

As a follow-up project, a more profound analysis of ergodic properties of random walks under resetting in various network types 
should be carried out. In this context, infinite graph scaling limits appear worthy of consideration.
In the present paper we considered undirected networks. However, this approach can readily be adapted to 
directed graphs in order to investigate first hitting features of biased motions under resetting or of motions with modified navigation rules \cite{Fronczak2009}.
The results presented in this paper have a large potential of generalization with a wide range of 
possible interdisciplinary applications
\cite{Evans-etal2020}. 
On the other hand, in the framework of the present approach a further investigation of the MFHT for non-Markovian RRPs
in case of fixed t-node sets and variable r-node populations in various network types is of interest.

A further interesting direction is to endow the walker with a memory of previously visited nodes to discourage or encourage visits of already explored regions \cite{BoyerSolisSalas2014,MeyerRieger2021}.
Moreover, the entries of the resetting matrix $R_{ij}'$ may depend on the walker's position at the time of the reset. This problem has a connection with partial resetting \cite{Tal-Friedman-etal2022}. 
A further topic calling for thorough investigation is transient reset. In this situation, the time intervals between consecutive resets are defective random variables, and the governing RRP comprises
a finite number of resetting events almost surely \cite{Dono_Mi_Po_Ria2024}. In such a model, the number of excursions is ``naturally'' limited
by the (energy) cost of the search process \cite{PalSandev2023}. 
Finally, one may consider also telegraph processes \cite{SRW2023,Sandev_Iomin_PRE2024} or some kinds of
deterministic motions under resetting.

\section*{Acknowledgements}
We thank the two anonymous reviewers for their valuable comments, which helped us to improve the presentation and interpretation of some results.
The authors G.D.\ and F.P.\ thank the GNAMPA
group of INdAM and acknowledge financial support under the MIUR-PRIN 2022 project ``Non-Markovian dynamics and non-local equations'', no. 202277N5H9
- CUP: D53D23005670006. 

%
%


\appendix
%
%
\setcounter{equation}{0}
\renewcommand{\theequation}{A\arabic{equation}}
\section*{Appendix}
\section{Proof of formulas (\ref{represntation_Tij}) and (\ref{mean_relax})}
\label{appendix_num_eval}
Let ${\rm Diag}(\mathbf{M}) = [\delta_{ij} M_{jj}] $ be the matrix containing the diagonal elements of $\mathbf{M} =[M_{ij}]$
and consider $\mathbf{H}$ the all-ones matrix, i.e.\ $H_{ij}=1$ for all $i,j$. Then, the MFPT matrix (\ref{represntation_Tij}) is given by
\beq
\label{MFPT_computation}
\big\langle \mathbf{T}(p;\mathbf{R}) \big\rangle = 
\left( \mathbf{1} + \mathbf{H} \cdot {\rm Diag}(\mathbf{S}) -\mathbf{S}\right) \cdot 
\left({\rm Diag} \left[ p\cdot\mathbf{R}\cdot \left(\mathbf{1}-q\mathbf{W}\right)^{-1} \right] \right)^{-1}
\eeq
where $\mathbf{S}$ is defined in (\ref{S-matrix}). The condition $q=1-p<1$ is required for $\mathbf{1}-q\mathbf{W}$ to be invertible.

Now we derive (\ref{mean_relax}).
From (\ref{Bernoulli_GF_prop}) and (\ref{NESS_Bernoulli}) one has straightforwardly
\beq
\label{Pinftyuone}
{\bar {\bf P}}(u) -\frac{\mathbf{P}(\infty)}{1-u} = \left(\mathbf{1}-q\mathbf{W}-p\mathbf{R}\right) \cdot 
\frac{\mathbf{1}}{(\mathbf{1}-q\mathbf{W}) \cdot (\mathbf{1}-qu\mathbf{W})}
\eeq
Setting $u=1$ 
yields (\ref{mean_relax}) and (\ref{Kemeny_res}).
\section{First hitting path statistics}
\label{fist_encouter_paths}
Here we aim to motivate our approach of killing the walker once t-nodes are reached and to
establish the connection of the first passage probability $F_{ib}(t)$ in Eq.\ 
(\ref{first-passage_proba}) with the FHT PDF of (\ref{MFPD}) for Markovian walks
with a single t-node $b$. An extension to multiple t-nodes is briefly discussed at the end of this section.
Consider a sample path of $t$ steps, $t\geq 1$, starting from a node $i$, passing through the nodes
$r_1 \to r_2 \to \ldots \to r_t$ and introduce an indicator function (IF), telling us
whether the target node $b$ is hit or not.
This survival IF has then the form
\beq
\label{first_hit_r}
h(r_1, \ldots, r_t; b) = (1-\delta_{r_1,b})\ldots (1-\delta_{r_t,b}), \hspace{1cm} t= 1,2,\ldots , \hspace{0.5cm} r_j = 1, \ldots, N.
\eeq
For $t=0$ we define $h = 1$. The function $h(r_1, \ldots, r_t; b)=1$ while the walker is alive (i.e.\ it has not hit $b$ yet) and $h(r_1, \ldots, r_t; b)=0$ otherwise.
In addition, the first hitting time IF reads
\beq
\label{IF_first}
f(r_1, \ldots, r_t; b) = h(r_1, \ldots, r_{t-1}; b) \delta_{r_t,b} , \hspace{1cm} t=1,2,\ldots , \hspace{0.5cm} r_j = 1, \ldots, N
\eeq
with $f(r_1, \ldots, r_t; b) = 1$ at the time instant when $b$ is hit for the first time and zero otherwise. It makes sense to define $f = 0$ for $t=0$.
Observe that for $t\geq 1$ one has
\beq
\label{important_feature}
h(r_1, \ldots, r_{t-1}; b) - f(r_1, \ldots, r_t; b) = h(r_1, \ldots, r_{t-1}; b) (1- \delta_{r_t,b}) = h(r_1, \ldots, r_t; b) 
\eeq
and
\beq
\label{unity_feature}
1 =  (1-\delta_{r_1,b} +\delta_{r_1,b}) \ldots  (1-\delta_{r_t,b} +\delta_{r_t,b}) =h(r_1, \ldots, r_{t}; b)  + \sum_{k=1}^{t} f(r_1, \ldots, r_k; b) 
\eeq
where $f(r_1, \ldots, r_k; b)$ considers the sample paths hitting $b$ at step $k$ for the first time, no matter whether at the remaining $t-k$ steps $b$ is hit or not.
Then, we consider the sample path average to hit the node $k$
in one step, given that the starting node is $i$, 
\beq
\label{cond_i_averge}
\langle \delta_{r,k} \rangle = \sum_{r=1}^N W_{ir} \delta_{r,k} = W_{ik} 
\eeq
which clearly recovers the transition matrix element.
The probability that the Markovian walker, starting from node $i$, chooses a specific sample path $i \to r_1 \to \ldots \to r_t$ is given by
$\Pi(i \to r_1 \to \ldots r_t)= W_{ir_1}W_{r_1,r_2} \ldots W_{r_{t-1},r_t}$.
We now average the survival IF over all possible paths of $t$ steps while keeping the starting node $i$ fixed to arrive at
\beq
\label{g_avergae}
\begin{array}{clr}
\ds \big\langle h(r_1, \ldots, r_t; b) \big\rangle & =  \ds \sum_{r_1=1}^N \ldots \sum_{r_t=1}^N \Pi(i \to r_1 \to \ldots, \to r_t) h(r_1,\ldots,r_t; b) & \\ \\ & = \ds \sum_{r_1=1}^N \ldots \sum_{r_t=1}^N  W_{ir_1}W_{r_1,r_2} \ldots W_{r_{t-1},r_t} h(r_1,\ldots,r_t; b) & \\ \\ & = \ds
\sum_{r_1=1}^N \ldots \sum_{r_t=1}^N {\widetilde W}_{ir_1} \ldots {\widetilde W}_{r_{t-1},r_t} = \sum_{r=1}^N
[{\widetilde {\bf W}}^t]_{ir}
= \Lambda_i(t) & 
\end{array}
\eeq
retrieving the survival probability (\ref{survival_monkey}).
Performing this average for (\ref{IF_first}) yields the probability to hit $b$ at step $t$ for the first time, or, in other words, the FHT PDF defined in (\ref{MFPD}). We will show below that the latter coincides with the first passage probability $F_{ib}(t)$ of (\ref{first-passage_proba}), as
\beq
\label{first_hitting_probability}
\big\langle f(r_1, \ldots, r_t; b) \big\rangle = \sum_{r_1=1}^N \ldots \sum_{r_{t-1}=1}^N  {\widetilde W}_{ir_1} \ldots {\widetilde W}_{r_{t-2},r_{t-1}} W_{r_{t-1},b} = [{\widetilde {\bf W}}^{t-1}\cdot \mathbf{W}]_{ib} = \chi_i(t) = F_{ib}(t)
\eeq
where we considered $\chi_i(1)= W_{ib}$ and $\chi_i(0)=0$.
This expression has an interesting interpretation: from $[{\widetilde {\bf W}}^{t-1}\cdot \mathbf{W}]_{ib} =
\sum_{k=1}^N[{\widetilde {\bf W}}^{t-1}]_{ik}W_{kb}$, the matrix element $[{\widetilde {\bf W}}^{t-1}]_{ik}$ denotes the probability that the walker arrives at time instant $t-1$ alive at node $k$, and then hops with probability $W_{kb}$ from $k$ to the target node $b$ where it is killed at time $t$.
It is then straightforward to obtain 
$$
\Lambda_i(t) = \Lambda_i(t-1) -\chi_i(t)= \sum_{r=1}^N [{\widetilde {\bf W}}^{t-1}]_{ir} (1-W_{rb}) = 
\sum_{r=1}^N\sum_{r=s}^N [{\widetilde {\bf W}}^{t-1}]_{ir} W_{rs}(1-\delta_{s,b}) = 
\sum_{s=1}^N [{\widetilde {\bf W}}^t]_{is} 
$$
which also corresponds to the averaging of (\ref{important_feature}).
The averaging of (\ref{unity_feature}) retrieves (\ref{survival_monkey}), and letting $t\to \infty$, with the aid of (\ref{first_hitting_probability}), takes us to 
\beq
\label{remarkable_relation}
\sum_{t=1}^{\infty} F_{ib}(t) =  \left[(\mathbf{1}-{\widetilde {\bf W}})^{-1} \cdot \mathbf{W}\right]_{ib} = 1 , 
\hspace{1cm} i =1 ,\ldots, N. 
\eeq
This tells us that $b$ is eventually hit with probability one, reconfirming the well-known feature of the recurrence of Markovian walks in finite (ergodic) networks 
\cite{MiRia2017,fractionalbook-2019,Kac1947,Lefebvre2007}. Further, consider that, in order to arrive at (\ref{remarkable_relation}), we implicitly used $\rho({\widetilde {\bf W}}) <1$
as $\Lambda_i(\infty)=0$ for all $i=1,\ldots,N$.
We can easily double check the  correctness of (\ref{remarkable_relation}):
$$
\sum_{k=1}^N \sum_{r=1}^N[\mathbf{1} -{\widetilde {\bf W}}]_{ik}[\mathbf{1}-{\widetilde {\bf W}})^{-1}]_{kr}W_{rb} =
\sum_{r=1}^N \delta_{ir}W_{rb} = W_{ib} = 1-\sum_{k=1}^N{\widetilde W}_{kb} .
$$
Then, recall the MFHT to node $b$. Considering  (\ref{first_hitting_probability}) we have
\beq
\label{MFHT_b}
\begin{array}{clr}
\ds \langle T_{ib} \rangle & = \ds \sum_{t=1}^{\infty} t \, \chi_i(t) =    \ds \sum_{t=1}^{\infty} t \, [{\widetilde {\bf W}}^{t-1}\cdot \mathbf{W}]_{ib} & \\ \\
 & = \ds [(\mathbf{1}-{\widetilde {\bf W}})^{-2}\cdot \mathbf{W}]_{ib} = \sum_{k=1}^N \hspace{0.25cm} [(\mathbf{1}-{\widetilde {\bf W}})^{-1}]_{ik} \, [(\mathbf{1}-{\widetilde {\bf W}})^{-1}\cdot \mathbf{W}]_{kb}  & \\ \\
 & = \ds [(\mathbf{1}-{\widetilde {\bf W}})^{-1}]_{i} &
 \end{array}
\eeq
which, considering also (\ref{remarkable_relation}), recovers our previous result (\ref{MFPT_non_Markov}).

Finally, to establish the connection with (\ref{first-passage_proba}), 
we need the conditional sample path average of $t$ steps hitting $b$ at the step $k \leq t$ for the first time 
and ending again on node $b$ 
\beq
\label{cond_av}
\big\langle f(r_1, \ldots, r_k; b) \, \delta_{r_t,b} \big\rangle = [{\widetilde {\bf W}}^{k-1}\cdot \mathbf{W}]_{ib} 
[\mathbf{W}^{t-k}]_{bb} =  F_{ib}(k) [\mathbf{W}^{t-k}]_{bb}  , \hspace{1cm} t \in \mathbb{N}.
\eeq
Here we have used that $[\mathbf{W}^{t-k}]_{bb} =$ is the conditional sample path average of the indicator function of the transitions $b \to b$ in $t-k$ steps, corresponding to the Markov property.
The summation of (\ref{cond_av}) from $k=1$ to $t$ covers all sample paths starting from $i$, ending on $b$ of $t$ steps,
containing $b$ at least once. The resulting quantity can therefore be identified with the transition probability from $i$ to $b$ in $t$ steps:
\beq
\label{final_eq}
[\mathbf{W}^t]_{ib} = \sum_{k=1}^t F_{ib}(k) [\mathbf{W}^{t-k}]_{bb} , \hspace{1cm} t \geq 1 
\eeq
holding for Makovian walks.
We identify this equation with (\ref{first-passage_proba}) 
if we use $F_{ib}(0)=0$ and the initial condition $[\mathbf{W}^0]_{ib} =\delta_{ib}$.

These considerations show that the first hitting statistics
with the classical approach 
of Section \ref{first_passage} and that with killing t-nodes of Section \ref{FPH-quantities} are equivalent for Markovian walks. Note, the latter does not require the Markov property of the walk and can be appropriately adapted to a non-Markovian resetting framework.

\subsection{Arbitrary t-node sets}
Consider now a set ${\cal B}$ of (multiple) t-nodes.
The survival IF of a sample path of $t$ steps $r_1 \to r_2 \to \ldots \to r_t$ then writes 
\beq
\label{suvival_of_walker}
h(r_1, \ldots , r_t ; {\cal B}) =\prod_{b\in {\cal B}} h(r_1,\ldots,r_t; b) = 
(1- \Theta(r_1,{\cal B}))\ldots  (1- \Theta(r_t,{\cal B})) , \hspace{1cm} t \geq 1
\eeq
where one has $1-\Theta(r,{\cal B}) = \prod_{b \in {\cal B}} (1-\delta_{r,b})$ and recall $\Theta(r, {\cal B}) = \sum_{b\in {\cal B}} \delta_{r,b}$.
The FHT IF then reads
\beq
\label{FHT_IF_multiple_nodes}
f(r_1, \ldots , r_t ; {\cal B}) = h(r_1, \ldots , r_{t-1} ; {\cal B}) \Theta(r_t,{\cal B})
\eeq
where $h(r_1, \ldots , r_{t-1} ; {\cal B}) - f(r_1, \ldots , r_t ; {\cal B}) = h(r_1, \ldots , r_t ; {\cal B})$ is the counterpart to (\ref{important_feature}). With these ingredients, the appropriate distributions governing the first hitting statistics of the t-node set ${\cal B}$ can straightforwardly be derived. 
\section{Some derivations related to the SP}
\label{AWR_derivations}
Let us prove the connection of the GF (\ref{Gf_transition matrix}) with the GF of the SP 
(\ref{Gf_transition matrix_AW}). Consider the matrix function (\ref{Gf_transition matrix_AW}) 
for the arguments $\mathbf{W},\mathbf{R}$ and evaluate
\beq
\label{derive}
\begin{array}{clr}
\ds \frac{{\bf 1}}{ \mathbf{1} -  u {\bar g}(u\mathbf{W}) \cdot \mathbf{R}} &= \ds \frac{{\bf 1}}{ \mathbf{1} - 
 \mathbf{R} {\bar \psi}(u)} = \ds \mathbf{1}+ \frac{ \mathbf{R} {\bar \psi}(u)}{ \mathbf{1} - 
 \mathbf{R} {\bar \psi}(u)}  & \\ \\
  & = \ds \mathbf{1} + \mathbf{R} \sum_{n=1}^{\infty} [{\bar \psi}(u)]^n = \mathbf{1} + \frac{{\bar \psi}(u)}{1-{\bar \psi}(u)}\mathbf{R} &
  \end{array}
\eeq
as $ u {\bar g}(u\mathbf{W}) \cdot \mathbf{R} = {\bar \psi}(u) \mathbf{R}$ and $\mathbf{R}^n=\mathbf{R}$ for $n \geq 1$. Hence, 
we retrieved (\ref{Gf_transition matrix}).

Now let us derive the renewal equation (\ref{renewal_AWR}). To that end we rewrite the GF (\ref{Gf_transition matrix_AW}) as follows:
\beq
\label{rewrite_AWR_GF}
{\bar {\bf P}}_{\text{AWR}}({\widetilde {\bf W}}, {\widetilde {\bf R}} ;u )  =  
\frac{{\mathbf 1} -  {\bar \psi}(u {\widetilde {\bf W}})}{{\mathbf 1} - u{\widetilde {\bf W}} } +
 u {\bar g}(u\mathbf{W}) \cdot  {\widetilde {\bf R}}
\cdot {\bar {\bf P}}_{\text{AWR}}({\widetilde {\bf W}}, {\widetilde {\bf R}} ;u )
\eeq
Inverting this relation we obtain (\ref{renewal_AWR}).
It is also instructive to consider (\ref{Gf_transition matrix_AW}) in the case of geometric resets. This yields
\beq
\label{Bernoulli_AWR}
{\bar {\bf P}}_{\text{AWR}}({\widetilde {\bf W}}, {\widetilde {\bf R}} ;u ) = [\mathbf{1} -u(q{\widetilde {\bf W}} +p {\widetilde {\bf R}})]^{-1}
\eeq
and this leads to
\beq
\label{Bernoulli_reset_AWR_propagator}
{\bf P}_{\text{AWR}}({\widetilde {\bf W}}, {\widetilde {\bf R}} ; t ) = (q{\widetilde {\bf W}} +p {\widetilde {\bf R}})^t
\eeq
which shows Markovianity of the process. By replacing ${\widetilde {\bf W}}$ and ${\widetilde {\bf R}}$ with $\mathbf{W}$ and $\mathbf{R}$ in (\ref{Bernoulli_reset_AWR_propagator}) we retrieve the propagator (\ref{equival}).

\section{Sibuya distributed resetting time intervals}
\label{standard_Sib}
Here we are concerned of the class of walks in which the time between consecutive resets is fat-tailed.
In particular, we consider the case of a resetting PDF characterized by a power-law tail:
$\psi_{\alpha}(t) \propto t^{-\alpha-1}$ for large $t$ and exponent $\alpha \in (0,1)$. 
As a prototypical example of this class, we consider Sibuya distributed times between consecutive resets. The Sibuya PDF reads 
\cite{PachonPolitoRicciuti2021,Sibuya1979}
\beq
\label{Sibuya}
\psi_{\alpha}(t) = \frac{(-1)^{t-1}}{t!} \alpha(\alpha-1)\ldots (\alpha-t+1) = 
\frac{\alpha \Gamma(t-\alpha)}{\Gamma(1-\alpha)\Gamma(t+1)}
,\hspace{0.5cm} t=1,2,\ldots, \qquad \alpha \in (0,1)
\eeq
with $\psi_{\alpha}(0)=0$
and
$\ds \psi_{\alpha}(t) \sim \alpha t^{-\alpha-1}/\Gamma(1-\alpha) $,  $t\to \infty$.
It has GF
\beq
\label{Sibuya_GF}
{\bar \psi}_{\alpha}(u)=1-(1-u)^{\alpha} , \hspace{1cm} |u| \leq 1 .
\eeq
The probability $\Phi_{\alpha}^{(0)}(t)$ that no reset happens up to and including time $t$
has GF ${\bar \Phi}_{\alpha}^{(0)}(u)= (1-u)^{\alpha-1}$, $|u|< 1$, and reads
\beq
\label{persistence}
\Phi_{\alpha}^{(0)}(t) = \frac{\Gamma(t+1-\alpha)}{\Gamma(1-\alpha)\Gamma(t+1)} \qquad \text{with }\Phi_{\alpha}^{(0)}(t)\sim \frac{t^{-\alpha}}{\Gamma(1-\alpha)} ,\hspace{0.5cm}  t \to \infty.
\eeq
Further, the double GF (\ref{GF_FTV}) reads
\beq
\label{Sibuya_double_f}
{\bar f}_{\alpha}(u,v) = (1-u)^{-\alpha}(1-uv)^{\alpha-1} .
\eeq
Its inversion with respect to $v$ gives $(1-u)^{-\alpha} u^b \Phi_{\alpha}^{(0)}(b)$.
Then, setting ${\bar h}_{\alpha}(u) = (1-u)^{-\alpha}$, we obtain 
$h(_{\alpha}t) = \Gamma(t+\alpha)/[(\Gamma(\alpha)\Gamma(t+1)]$
leading to the formula (see (\ref{ftB}))
\beq
\label{f-alpha-t-B}
f_{\alpha}(t,b) = h_{\alpha}(t-b) \Phi_{\alpha}^{(0)}(b) = \frac{\Gamma(t-b+\alpha)}{\Gamma(\alpha)\Gamma(t-b+1)} \, \frac{\Gamma(b+1-\alpha)}{\Gamma(1-\alpha)\Gamma(b+1)} , \hspace{1cm} b  \leq t
\eeq
with $f_{\alpha}(t,b)=0$ for $b>t$ (as $h_{\alpha}(\tau)$ is causal and supported on 
$\mathbb{N}_0$, which makes sense as $B_{n,t}=t-J_n \leq t$). Note, we retrieve (\ref{persistence}) for $b=t$. 
Using causality of $h_{\alpha}(\tau)$ we can write down the GF of (\ref{f-alpha-t-B}) 
with respect to $b$, obtaining
\beq
\label{GF_V}
{\bar f}_{\alpha}(t,v) = \sum_{b=0}^t h_{\alpha}(t-b) \Phi_{\alpha}^{(0)}(b) v^b = 
\sum_{b=0}^t \frac{\Gamma(t-b+\alpha)}{\Gamma(\alpha)\Gamma(t-b+1)} \, \frac{\Gamma(b+1-\alpha)}{\Gamma(1-\alpha)\Gamma(b+1)} v^b .
\eeq
The expansion stops at $b=t$ and is covered by (\ref{Gf_B}) if we recognize that 
$h_{\alpha}(t) = \delta_{t0}+ {\cal R}_{\alpha}(t)$,
where ${\cal R}_{\alpha}(t)$ is the resetting rate of the Sibuya RRP (with ${\cal R}_{\alpha}(0)=0$), see formula \eqref{resettying_sib} below.
Interesting is the asymptotics $t \gg b \gg 1$:
\beq
\label{Dynkin}
f_{\alpha}(t,b) \sim \frac{t^{\alpha-1}}{\Gamma(\alpha)} \frac{b^{-\alpha}}{\Gamma(1-\alpha)}.
\eeq
This relation coincides with the one reported for continuous time: see 
\cite{Barkai-etal-2023} (Eq.\ (12) in that paper --- recall Euler’s 
reflection formula for the gamma function)
and \cite{GodrecheLuck2001,Dynkin1955}.
The Sibuya resetting rate (see (\ref{resetting_rate})) has GF
\beq
\label{resettying_sib_GF}
{\bar {\cal R}}_{\alpha}(u) = (1-u)^{-\alpha} -1 = {\bar h}_{\alpha}(u)- 1 .
\eeq
Its inversion leads to the expression
\beq
\label{resettying_sib}
 {\cal R}_{\alpha}(t) = h_{\alpha}(t) -\delta_{t0}  = -\delta_{t0} + \frac{\Gamma(\alpha+t)}{\Gamma(\alpha)\Gamma(t+1)}
\eeq
with ${\cal R}_{\alpha}(0) =0$ and the large time asymptotics ${\cal R}_{\alpha}(t) \sim t^{\alpha-1}/\Gamma(\alpha)$. The dependence of ${\cal R}_{\alpha}(t)$ with respect to $\alpha$ is shown in 
Fig.\ \ref{Sib_resetting_rate} for some values of $t$. Note that $0 \leq {\cal R}_{\alpha}(t) \leq 1$ which follows by considering that 
${\cal R}_{\alpha}(t) = \langle {\cal N}_{\alpha}(t)-{\cal N}_{\alpha}(t-1) \rangle $ where ${\cal N}_{\alpha}(t)$ denotes the Sibuya counting renewal process. Recall, we have
 ${\cal N}_{\alpha}(t)-{\cal N}_{\alpha}(t-1) \in \{0,1\}$,  $t\geq 1$, which holds true for any discrete-time renewal resetting rate.
In particular, for $\alpha \to 1-$, one has ${\cal R}_{1-}(t) =1$,  $t\geq 1$ corresponding to the trivial counting process where ${\cal N}_{\alpha}(t)=t$ almost surely.

\section{Comparison of the AW with absorbing Markov chains}
\label{AW_AMC}
Here we establish the connection of the AW where t-nodes are killing the walker with that of a corresponding
absorbing Markov chain (AMC) where the walker is absorbed (trapped) once reaching a t-node.
We emphasize that we only consider the Markovian setting in which we have either no resetting, then the transition matrix
is $\mathbf{W}$ or the case of Bernoulli resetting where the transition matrix of the associated walk becomes $q\mathbf{W}+p\mathbf{R}$.

Consider first the auxiliary walk of the killing t-nodes approach.
Then we have for the AW transition matrix (\ref{transition_mat_mod}) the representation
\beq
\label{absorbing_chain}
{\widetilde {\bf W}} = \left[ \begin{array}{cl} \mathbf{Q} \,\, ; & \mathbf{0} \\
\mathbf{q} \,\, ; & \mathbf{0} \end{array} \right]
\eeq
where we changed the node labeling such that $i=1,\ldots , N-n_b$ 
label the nodes which are not t-nodes (the transient nodes in the AMC picture) and $i=N-n_b+1,\ldots , N$ label the t-nodes where $n_b$ denotes the number of t-nodes. 
In (\ref{absorbing_chain}) we introduced the
$ (N-n_b) \times  (N-n_b)$ matrix $\mathbf{Q}$ defined by
\beq
\label{Q_matrix}
\mathbf{Q} = [\mathbf{1} - \mathbf{\Theta}({\cal B})] \cdot \mathbf{W} \cdot [\mathbf{1} - \mathbf{\Theta}({\cal B})] 
\eeq
where $\mathcal{B}$ is the set of t-nodes.
This matrix is generated by setting to zero all columns and rows of $\mathbf{W}$
which label t-nodes.
The matrix $\mathbf{q}$ is a $n_b \times (N-n_b)$ matrix and is given by
\beq
\label{small_q_mat}
\mathbf{q}= \mathbf{\Theta}({\cal B}) \cdot \mathbf{W} \cdot [\mathbf{1} - \mathbf{\Theta}({\cal B})].
\eeq
The matrix elements of $\mathbf{Q}$ constitute the part of ${\widetilde {\bf W}}$ for which the departure nodes are not t-nodes, and $\mathbf{q}$ the part for which they are t-nodes. 
Then, the AW transition matrix for $t$ steps is 
\beq
\label{absorbing_chain_t_steps}
{\widetilde {\bf W}}^t = {\widetilde {\bf W}} \cdot \mathbf{Q}^{t-1}=\left[ \begin{array}{cl} \mathbf{Q}^t \,\, ;& \mathbf{0} \\
\mathbf{q}\cdot \mathbf{Q}^{t-1} \,\, ;& \mathbf{0} \end{array} \right] ,\hspace{1cm} t \geq 1 .
\eeq
The matrix element $[\mathbf{Q}^t]_{ij}$ denotes the probability that the walker arrives alive at time $t$ on node $j$  where neither the departure node
$i$ nor the arrival node $j$ are t-nodes. On the other hand, $[\mathbf{q}\cdot 
\mathbf{Q}^{t-1}]_j$ is the probability that the walker arrives alive at $j \notin {\cal B}$ but having started its walk from a t-node.

The survival probability of the walker at time $t$ is $\Lambda_i(t)=\sum_{j=1}^N[\mathbf{Q}^t]_{ij}$ for departure nodes which are not t-nodes $i \notin {\cal B}$,  and
$\Lambda_i(t) = \sum_{j=1}^N[\mathbf{q}\cdot \mathbf{Q}^{t-1}]_j $ for departure nodes that are t-nodes $i \in {\cal B}$. 
This includes (see Eq. (\ref{MFPT_non_Markov})) the MFPT in which the matrix
\beq
\label{MFPT_comp}
[\mathbf{1} -{\widetilde {\bf W}}]^{-1} = 
\left[ \begin{array}{cl} [\mathbf{I}'-\mathbf{Q}]^{-1} \,\, ;& \mathbf{0} \\
\mathbf{q}\cdot [\mathbf{I}'-\mathbf{Q}]^{-1} \,\, ; & \mathbf{I} \end{array} \right]
\eeq
considers all departure nodes (including t-nodes). We denote here with $\mathbf{I}'$ the $(N-n_b)\times (N-n_b)$ identity matrix.
Keep also in mind that (\ref{absorbing_chain_t_steps}) is defective with $\Lambda_i(\infty)=0$.
Now, we compare this setting with the AMC approach.
 The transition matrix of the corresponding AMC has the canonical form \cite{GrinsteadSnell2006}
\beq
\label{AMC_trans_mat}
\mathbf{V} = \left[ \begin{array}{cl} \mathbf{Q} \,\, ;& \mathbf{r} \\
\mathbf{0} \,\, ; & \mathbf{I} \end{array} \right]
\eeq
where $ \mathbf{Q}$ is again defined by (\ref{Q_matrix}) and 
$\mathbf{r} = [\mathbf{1} - \mathbf{\Theta}({\cal B})] \cdot \mathbf{W} \cdot \mathbf{\Theta}({\cal B})$ contains the probabilities 
to hop from nodes which are not t-nodes to t-nodes 
in one step.  $\mathbf{I}$ is the $n_b \times n_b$ identity matrix which ensures that the walker is absorbed
once hopping on a t-node. Differently from (\ref{absorbing_chain}), 
the AMC transition matrix is per construction row-stochastic. The AMC transition matrix of $t$ steps reads
\beq
\label{t_step_AMC_trans_mat}
\mathbf{V}^t = \left[ \begin{array}{cl} \mathbf{Q}^t \,\, ;& \left(\mathbf{1}- \mathbf{Q}\right)^{-1}\cdot\left(\mathbf{1}- \mathbf{Q}^{\, t}\right) \cdot \mathbf{r} \\ 
\mathbf{0} \,\, ;& \hspace{2cm} \mathbf{I} \end{array} \right] , \hspace{1cm} t \geq 0
\eeq
and contains the matrix $\mathbf{Q}^t$ but not the matrix $[\mathbf{q}\cdot \mathbf{Q}^{t-1}]$ of (\ref{absorbing_chain_t_steps}) which captures the first hitting statistics for departures from t-nodes. 
Moreover, the AMC approach ignores the statistics of the first hits with departures from t-nodes such that the statistics of first returns to t-nodes: an information which is captured by the approach of killing t-nodes.

\section{Derivation of (\ref{ftB}) and (\ref{Gf_B})}
\label{derivations_18_19}
Here we derive (\ref{ftB}) and (\ref{Gf_B}) in details. 
We start by considering the double GF 
${\bar f}(u,v) = \sum_{t=0}^{\infty} \sum_{b=0}^{\infty} f(t,b) u^tv^b$ which is given by
$$
{\bar f}(u,v)= \frac{1}{1-{\bar \psi}(u)} \frac{1-{\bar \psi}(uv)}{1-uv} = 
\left( 1 + \frac{ {\bar \psi}(u) }{1-{\bar \psi}(u)}   \right)  \frac{1-{\bar \psi}(uv)}{1-uv} =\left(1+{\bar R}(u)\right){\bar \Phi}^{(0)}(uv)
$$
where ${\bar \Phi}^{(0)}(u)$ is GF of the persistence probability $\Phi^{(0)}(t)$ and ${\bar {\cal R}}(u)$ of the resetting rate defined by (\ref{resetting_rate}).
Then, the inversion of ${\bar \Phi}^{(0)}(u v)$ with respect to $v$ yields $ u^b\Phi^{(0)}(b)$. Hence, the inversion of ${\bar f}(u,v)$ with respect to
$v$ gives
$$
{\tilde {\bar f}(}u,b) = \left(1+{\bar {\cal R}}(u)\right) u^b \Phi^{(0)}(b).
$$
The inversion with respect to $u$ of $u^b{\bar h}(u)$ leads to $h(t-b)$. Therefore, inverting 
$\left(1+{\bar R}(u)\right) u^b$ with respect to $u$ takes us to $\delta_{tb}+{\cal R}(t-b)$. Finally, inverting
${\tilde {\bar f}}(u,b)$ with respect to $u$ yields to Eq. (\ref{ftB}):
$$
f(t,b) = \left(\delta_{tb}+{\cal R}(t-b)\right) \Phi^{(0)}(b) .
$$
To obtain (\ref{Gf_B}), we need to account for ${\cal R}(\tau) = 0$  for $\tau \leq 0$  
which follows from the causality of $\psi(\tau)$ which is null
for $\tau \leq 0$. Keeping this in mind we take from $f(t,b)$ the GF with respect to $b$ 
which takes us to Eq. (\ref{Gf_B}): 
$$
\bar{f}(t,v) = \sum_{b=0}^t \left(\delta_{tb}+{\cal R}(t-b)\right) \Phi^{(0)}(b) v^b =
 \Phi^{(0)}(t) v^t +  \sum_{b=0}^t  {\cal R}(t-b) \Phi^{(0)}(b) v^b .
$$
%
%

%
%

%
%

\end{document}